\documentclass[10pt]{amsart}

\usepackage{graphicx}
\usepackage{amsfonts}
\usepackage{amsmath}
\usepackage{amssymb}
\usepackage{amsopn}
\usepackage{mathrsfs}
\usepackage{enumerate}
\usepackage[hidelinks]{hyperref}

\DeclareMathOperator*{\support}{supp}

\DeclareMathOperator{\interior}{int}

\DeclareMathOperator*{\degree}{deg}

\DeclareMathOperator*{\Reg}{\mathbf{Reg}}

\DeclareMathOperator*{\capa}{cap}
\DeclareMathOperator*{\dist}{dist}

\DeclareMathOperator*{\Lip}{Lip}
\DeclareMathOperator*{\re}{Re}

\DeclareMathOperator*{\Prob}{Prob}
\DeclareMathOperator*{\sing}{Poles}
\DeclareMathOperator*{\spann}{span}

\renewcommand{\wp}{\mathscr{P}}
\renewcommand{\Re}{\re}

\renewcommand{\epsilon}{\varepsilon}

\newcommand{\wconverge}[1]{\rightharpoonup^*{#1}}

\newcommand{\C}{\mathbb{C}}

\newcommand{\N}{\mathbb{N}}
\newcommand{\M}{\mathcal{M}}

\newcommand{\bs}[1]{\boldsymbol{#1}}
\renewcommand{\phi}{\varphi}
\newcommand{\z}{\zeta}
\newcommand{\nalign}[1]{\begin{equation}\begin{split} #1 \end{split}\end{equation}}
\newcommand{\nalignstar}[1]{\begin{equation*}\begin{split} #1 \end{split}\end{equation*}}

\newtheorem{theorem}{Theorem}[section]
\newtheorem{corollary}{Corollary}[theorem]
\newtheorem{proposition}{Proposition}[section]

\newtheorem{definition}[theorem]{Definition}
\newtheorem{remark}[theorem]{Remark}

\newcounter{theexample} 
    \newenvironment{example}[1][]{\textbf{Example \arabic{theexample}. #1 \stepcounter{theexample}\rm}{}}

\begin{document}
\begin{abstract}
The Bernstein Markov Property, is an asymptotic quantitative assumption on the growth of uniform norms of polynomials or rational functions on a compact set with respect to $L^2_\mu$-norms, where $\mu$ is a positive finite measure. We consider two variants of the Bernstein Markov property for rational functions with restricted poles and compare them with the polynomial Bernstein Markov property finding out some sufficient conditions for the latter to imply the former. Moreover, we recover a sufficient \emph{mass-density} condition for a measure to satisfy the rational Bernstein Markov property on its support. Finally we present, as an application, a meromorphic $L^2$ version of the Bernstein Walsh Lemma.  
\end{abstract}

\author{Federico Piazzon}
\address{room 712 Department of Mathematics, Universit\'a di Padova, Italy. 
Phone +39 0498271260}
\email{\underline{fpiazzon@math.unipd.it}} 
\urladdr{http://www.math.unipd.it/~fpiazzon/   (work in progress)}
\thanks{Dep. of Mathematics University of Padua, Italy. Supported by Doctoral School on Mathematical Sciences.}
\date{\today}
\title[Rational BMP]{Some results on the rational Bernstein Markov Property in the complex 
plane}

\maketitle
\small
\begin{flushright}
\emph{This paper is dedicated to the memory of Professor Giuseppe Zampieri.}
\end{flushright}
\normalsize
\tableofcontents

\section{Introduction}
Let $K\subset \C$ be compact and have infinitely many points. In such a case $\|p\|_K:=\max_{z\in K}|p(z)|$ is a norm on the space $\wp ^k$ of polynomials of degree not greater than $k$ for any $k\in \N.$

Let us pick a positive finite Borel measure $\mu$ supported on $K$. When $\|\cdot\|_{L^2_\mu(K)}$ is a norm on $\wp^k$ we can compare it with the uniform norm on $K$. In fact, since $\wp^k$ is a finite dimensional normed vector space, there exist positive constants $c_1,c_2$ depending only on $(K,\mu,k)$ such that
$$c_1 \|p\|_{L^2_\mu}\leq \|p\|_K\leq c_2 \|p\|_{L^2_\mu}\;\;\forall p\in \wp^k.$$  

Notice that there exists such a $c_1$ because the measure $\mu$ is finite (one can take $c_1=\mu(K)^{-1/2}$) while $c_2$ is finite precisely when $\mu$ induces a norm.

The Bernstein Markov property is a quantitative asymptotic growth assumption on $c_2$ as $k\to \infty.$  Namely, the couple $(K,\mu)$ is said to enjoy the Bernstein Markov 
property if for any sequence $\{p_k\}: p_k\in \wp^k$ we have
\begin{equation}
 \limsup_k \left(\frac{\|p_k\|_{K}}{\|p_k\|_{L^2_\mu}}\right)^{1/k}\leq1.
 \label{BMP}
\end{equation}

The Bernstein Markov property can be equivalently defined in several complex variables and/or for weighted polynomials, i.e., functions of the type $pw^{\deg(p)}$ where $w$ is an admissible weight as in \cite{SaTo97}, see Definition \ref{wbmpdegf}.

We remark that the class of measures having the Bernstein Markov property is very close to the \textbf{Reg} class studied in the monograph \cite{StaTo92} (later generalized to the multidimensional case in \cite{B97}). Precisely, if we restrict our attention to measures $\mu$ whose support $\support \mu$ is a \emph{regular} set for the Dirichlet problem for the Laplace operator (i.e., $\C\setminus \support\mu$ admits a classical Green function $g$ with logarithmic pole at infinity such that $g|_{\partial \Omega_K}\equiv 0$, where $\Omega_K$ is the unbounded component of $\C\setminus K$) the two notions coincide.

The terminology Bernstein Markov property has been introduced in the framework of several complex variables and pluripotential theory; \cite{BB11b} and \cite{BlLe99}. In such a context it turns out that the Bernstein Markov property is a powerful tool in proving some deep results; see \cite{BB11}.    

In the present paper we prefer to deal with the Bernstein Markov property, though its name could be misleading, both because some of our proofs rely on the continuity of the Green function and because it looks more tailored to the proposed applications than the property of the class \textbf{Reg}.

A first motivation to the study of the Bernstein Markov property comes from approximation theory. If $(K,\mu)$ have the Bernstein Markov property then, given any holomorphic function $f$, the error of best polynomial approximation $p_k$ of degree not greater than $k$ to $f$ and the error of the approximation $q_k$ given by projection in $L^2_\mu$ on the subspace $\wp^k$ are asymptotically the same in the sense that for any $0<r<1$ and $f\in \mathscr C(K)$
$$\limsup_k \|f- p_k\|_K^{1/k}\leq r\;\;\text{ if and only if } \;\;\limsup_k \|f- q_k\|_{L^2_\mu}^{1/k}\leq r.$$
Consequently, one has a $L^2$ version of the Bernstein-Walsh Lemma \cite{Wal35} for Bernstein Markov measures relating the rate of best $L^2$ approximation of a function to its maximum radius of holomorphic extension; \cite[Prop. 9.4]{levnotes}. The several complex variables version of the Bernstein Walsh Lemma is usually referred as the Bernstein Walsh Siciak Theorem, see for instance \cite[Th. 9.7]{levnotes}.

Moreover, the Bernstein Markov property has been studied (see for instance \cite{BB11,BB11b,BLOLE07,B97,NqPh12}) in relation to (pluri-)potential theory, the study of plurisubharmonic functions in several complex variables. It turns out that such a property is fundamental both to recover the Siciak Zaharyuta extremal plurisubharmonic function and the (pluripotential) equilibrium measure (see \cite{levnotes}) by $L^2$ methods.

Lastly, Bernstein Markov measures play a central role in a recent theory of Large Deviation for random arrays and common zeroes of random polynomials; see for instance \cite{BlLe13,BlSh07} and references therein.

In the present paper we investigate two slightly modified versions of \eqref{BMP}. To do that 
we define the following classes of sequences of rational functions
\begin{align*}
\mathcal R(P)&:=\left\{\{p_k/q_k\}\,:\,p_k,q_k\in \wp^k,Z(q_k)\subseteq P\;\forall k\in \mathbb N \right\}\;\;\text{and}\\
\mathcal Q(P)&:=\left\{\{p_k/q_k\}\,:\,p_k,q_k\in \wp^k,\degree 
q_k=k,Z(q_k)\subseteq P \;\forall k\in \mathbb{N}\right\},
\end{align*}
where we set $Z(p):=\{z\in \C:\,p(z)=0\}$ and where $P\subset \C$ is any compact set that from now on we suppose to have empty intersection with $K$.

Throughout the paper we use the symbol $ \M^+(K)$ to denote the cone of positive Borel finite measures $\mu$ such that $\support \mu\subseteq K,$ adding a subscript $1$ for probability measures. 
\begin{definition}[Rational Bernstein Markov Property]\label{RBMPdef}
Let $K,P\subset \C$ be compact disjoint sets and $\mu\in \M^+(K).$
\begin{enumerate}[(i)]
\item \emph{(Rational Bernstein Markov Property.)} If
\begin{equation}
 \limsup_k\left(\frac{\|r_k\|_{K}}{\|r_k\|_{L^2_\mu}}\right)^{1/k}\leq 
1\;\;\forall \{r_k\}\in \mathcal R(P),
 \label{RBMP}
\end{equation}
then $(K,\mu,P)$ is said to enjoy the rational Bernstein Markov Property.
\item \emph{(sub-diagonal Rational Bernstein Markov Property.)} If
\begin{equation}
 \limsup_k\left(\frac{\|r_k\|_{K}}{\|r_k\|_{L^2_\mu}}\right)^{1/k}\leq 
1\;\;\forall \{r_k\}\in \mathcal Q(P),
 \label{SRBMP}
\end{equation}
then $(K,\mu,P)$ is said to enjoy the sub-diagonal rational Bernstein Markov property.
\end{enumerate}
\end{definition}

One motivation to study such properties is given by the discretization of a quite general class of vector energy problems performed in \cite{BlLeWi13}. Bloom, Levenberg and Wielonsky introduce a probability $\Prob(\cdot)$ on the space of sequences of arrays of points $\{\bs{z}^{(1)},\dots,\bs z^{(m)}\}$, where $\bs z^{(l)}=\{z_0^{(l)},\dots,z_k^{(l)}\}\in (K^{(l)})^{k+1}$, on a vector of compact sets $\{K^{(1)},\dots,K^{(m)}\}$ in the complex plane based on a vector of probability measures $\mu^{(i)}\in \mathcal M^+_1(K^{(i)})$ such that $(K^{(i)},\mu^{(i)},\cup_{j\neq i}K^{(j)})$ has the rational Bernstein Markov property. In \cite{BlLeWi13} the authors actually deal with \emph{strong} rational Bernstein Markov measures, which is a variant of rational Bernstein Markov property where weighted rational function are considered instead of standard ones, however their paper can be read in the \emph{un-weighted} setting picking (in their notation) $Q\equiv 0$. Then they prove a Large Deviation Principle (LDP) for measures canonically associated to arrays of points randomly generated according to $\Prob$. Also, they show that the validity of the LDP is not affected by the particular choice of $\{\mu^{(1)},\mu^{(2)},\dots,\mu^{(m)}\}$ that are only required to form a vector of rational Bernstein Markov measures.

Measures having the rational Bernstein Markov property are worth to be studied also from the approximation theory point of view. In fact, for such measures it turns out that the radius of maximum meromorphic extension with exactly $m$ poles of a function $f\in \mathscr C(K)$ is related to the asymptotic of its $L^2_\mu$ approximation numbers 
$$\left(\min_{\deg p\leq k,\deg q=m}\|f-p/q\|_{L^2_\mu}\right)^{1/k}.$$
The reader is referred to Section \ref{appl} for a precise statement. 

The paper is organized as follows. 

In \textbf{Section 2} we compare Definition \ref{RBMPdef} to the polynomial Bernstein Markov property. We address the following question. \emph{Are there sufficient additional conditions on }$(K,\mu,P)$ \emph{for the} polynomial Bernstein Markov property \emph{to imply the} rational Bernstein Markov property \emph{or the} sub-diagonal rational Bernstein Markov property\emph{?} 
A positive answer to both instances of such a question is given in Theorem \ref{mainresultsec2}, by means of an equivalent formulation of the problem suggested  in Propositions \ref{rationaltoweighted} and \ref{rationaltoweighted2}.

In \textbf{Section 3} we consider the classical $\Lambda^*$ condition (see \cite{StaTo92}) or mass density sufficient condition \cite{BlLe99} for the Bernstein Markov property. We notice (in Theorem \ref{PCcase})  that, due to Theorem \ref{mainresultsec2}, this condition implies the rational Bernstein Markov property as well, provided that $\hat K\cap P=\emptyset$ ($\hat K$ is the polynomial hull of $K$, see \eqref{polyhull}). 

In the case $\hat K\cap P\neq \emptyset, \hat P\cap K=\emptyset$  we show in Proposition \ref{techlemma} that it is possible to build a suitable conformal mapping $f$ such that the images $E$ of $K$ and $Q$ of $P$ under $f$ are in the relative position of the hypothesis of Theorem \ref{PCcase}. Thus, we derive (Theorem \ref{NPCcase}) a sufficient mass density condition for the rational Bernstein Markov property in a more general case. 

In \textbf{Section 4} we provide a uniform convergence result (see Theorem \ref{approximationtheorem}) for sequences of Green functions with poles in $P$ associated to compact subsets $K_j$ of $K$ whose logarithmic capacity (see \eqref{logcap} below) is converging to the one of $K$. Then we use such a theorem to give an alternative direct proof of the sufficient mass density condition of Theorem \ref{PCcase} .

In \textbf{Section 5} we present, as an application, a meromorphic $L^2_\mu$ version of the Bernstein Walsh Lemma; see Theorem 5. 

\section{Polynomial versus rational Bernstein Markov property}
Let us illustrate some significantly different situations which can occur by providing some easy examples where we are able to perform explicit computations.

We recall that, given an orthonormal basis $\{q_j\}_{j=1,2,\dots}$ of a separable Hilbert space $H$ (endowed with its induced norm $\|\cdot\|_H$) of continuous functions on a given compact set, the \emph{Bergman Function} $B_k(z)$ of the subspace $H_k:=\spann \{q_1,q_2,\dots,q_k\}$ is 
$$B_k(z):=\sum_{j=1}^k|q_j(z)|^2.$$
It follows by its definition and by Parseval Identity that for any function $f\in H_k$ one has
$|f(z)|\leq \sqrt{B_k(z)}\|f\|_{H},$ while the function $f(z):=\sum_{j=1}^k \bar q_j(z_0) q_j(z)$ achieves the equality at the point $z_0$, thus
\begin{equation}\label{extremal}B_k(z)=\max_{f\in H_k\setminus \{0\}}\left(\frac{|f(z)|}{\|f\|_H}\right)^2.
\end{equation}
\begin{example}
\begin{enumerate}[(a)]
\item Let $\mu$ be the arc length measure on the boundary $\partial \mathbb D$ 
of the unit disk. Let $K=\partial \mathbb D$ and $P=\{0\}$. 

Let us take a sequence $\{r_k\}=\left\{\frac{p_{l_k}}{z^k}\right\}$ in $\mathcal R(P)$ 
where $\degree p_{l_k}=l_k\leq k$, then we have
\begin{equation}\begin{split}\|r_k\|_{K}=\left\|\frac{p_{l_k}}{z^k}\right\|_K=\|p_{
l_k } \|_ { K } \leq \\
{\|B_{l_k}^{\mu}\|_{K}}^{1/2}\|p_{l_k}\|_{L^2(\mu)}={\|B_{l_k}^{\mu}\|_{
K } } ^ { 1/2 } \|r_k\|_ { L^2(\mu)}.\end{split}\label{bybergestimate}
\end{equation}
Here we indicated by $B_k^\mu(z)$ the Bergman function of the space 
$\left(\wp^k,\langle\cdot,\cdot\rangle_{L^2_\mu}\right)$.

For this choice of $\mu$ the orthonormal polynomials $q_k(z,\mu)$ are simply the normalized monomials $\left\{\frac{z^k}{\sqrt{2 \pi}}\right\}$, thus we have 
\begin{equation}
(\max_{K} B_{k}^\mu)^{1/2k}=\left(\max_{K} \frac{\sum_{j=0}^{k}|z|^{2j}}{2\pi}\right)^{1/2k}=\left(\frac{k+1}{2\pi}\right)^{1/2k}.\label{bergmax}
\end{equation}
It follows by \eqref{bybergestimate} and \eqref{bergmax} that $(K,\mu,P)$ has 
the rational Bernstein Markov Property. A similar computation shows that actually 
any $\nu$ such that $(K,\nu)$ has the Bernstein Markov Property 
is such that $(K,\nu,P)$ has the rational Bernstein Markov 
Property.
\item On the other hand, the same measure $\mu$ does not enjoy the sub-diagonal rational Bernstein Markov Property in the triple $(K,\mu,P)$ with  $K=\{1/2\leq|z|\leq 1\}$ and $P=\{0\}$ as the sequence of functions $\{1/z^k\}$ clearly shows: $\|z^{-k}\|_K=2^k,$ $\|z^{-k}\|_{L^2_\mu}=1.$ A fortiori the rational Bernstein Markov Property is not satisfied by $(K,\mu,P)$.
\item On the contrary, the arc length measure on the inner boundary of $K=\{1/2\leq|z|\leq 1\}$ and $P=\{0\}$ has the sub-diagonal rational Bernstein Markov Property, equation \eqref{SRBMP}, but neither the rational Bernstein Markov Property equation \eqref{RBMP}, nor the polynomial one, equation \eqref{BMP}, as is shown by the sequence $\left\{ z^k\right\}$. Notice that
\begin{align*}
&\left(\int_{\frac 1 2\partial\mathbb D}\left| z \right|^{2k} ds\right)^{1/2}=\sqrt{\pi}2^{-k} \text{ and } \|z^k\|_K=1, \text{ thus}\\
&\left(\frac{\|z^k\|_K } { \|z^k\|_ { L^2_\mu}}\right)^{1/k}=2\pi^{-1/2k}\to 2\nleq 1.
\end{align*}
In fact, in these last two examples the support of $\mu$ is not the whole set $K$, however we can provide a similar example also under the restriction $\support \mu= K.$
\item Let us take a dense sequence $\{z_j\}$ in $K=\{1/2\leq|z|\leq 1\}$ and a summable sequence of positive numbers $c:=\{c_j\}$ such that $\sum_{j=1}^\infty c_j=1$, we define 
$$\mu_c:=\frac 1{4\pi}ds|_{\partial \mathbb D}+\frac 1{2}\sum_{j=1}^\infty c_j\delta_{z_j}\in \mathcal M_1^+(K).$$ 
Notice that $\support \mu=K.$ It is well known that $ds|_{\partial \mathbb D}$ has the Bernstein Markov property for $\overline{ \mathbb D}$, so does the measure $\mu_c$ have.

On the other hand, we can show that $(K,\mu_c,\{0\})$ does not have the rational Bernstein Markov property, provided a suitable further assumption on $c$ and $z_j$.

Precisely, let $\{c_j\}\in \ell^1$ and a sequence $\{n_k\}$ of natural numbers be such that
\begin{align}
&\liminf_k \left(1+\sum_{j=k+1}^\infty c_j|z_j|^{2 n_k}  \right)^{1/2n_k}=1\nonumber\\
&0\leq k\leq n_k\label{sequenceassumption}\\
&\lim_k k/n_k <1.\nonumber
\end{align} 

We construct a sequence $\{\tilde r_{k}\}\in \mathcal Q(\{0\})$ of rational functions for which  \eqref{SRBMP} does not hold with $\mu=\mu_c$ and $P=\{0\}$; hence we show that $(K,\mu_c,P)$ does not have the sub-diagonal rational Bernstein Markov property.

Let us  define $r_{n_k}(z):=\frac{p_k(z)}{z^{n_k}}=\frac{\prod_{l=1}^kz-z_l}{z^{n_k}}.$ We notice that
\begin{align*}
\|r_{n_k}\|_K&=\max\left\{2^{n_k}\|p_k\|_{1/2\partial \mathbb D},\|p_k\|_{\partial \mathbb D}\right\}\geq 2^{n_k}\|p_k\|_{1/2\partial \mathbb D},\\
\|r_{n_k}\|_{L^2_{\mu_c}}&=\left(\frac 1{4\pi}\int_{\partial \mathbb D}|p_k|^2ds+\frac 1 2\sum_{j=k+1}^\infty\frac{c_j}{|z_j|^{2n_k}}|p_k(z_j)|^2\right)^{1/2}\\
&\leq \frac{\|p_k\|_{\partial \mathbb D}}{\sqrt 2}\left(1 +\sum_{j=k+1}^\infty\frac{c_j}{|z_j|^{2n_k}}\right)^{1/2}\\
&\leq 2^{-1/2+k} \|p_k\|_{1/2\partial \mathbb D}\left(1+\sum_{j=k+1}^\infty\frac{c_j}{|z_j|^{2n_k}}\right)^{1/2}.
\end{align*} 
Here we used the second equation in \eqref{sequenceassumption} and the classical Bernstein Walsh Inequality for $1/2\partial \mathbb D$ twice, e.g. $|p(z)|\leq \|p\|_{1/2\partial \mathbb D}\exp(\degree p \log^+(2|z|))$.
It follows that
$$\left(\frac{\|r_{n_k}\|_K}{\|r_{n_k}\|_{L^2_{\mu_c}}}\right)^{1/n_k}\geq 2^{1-\frac k{n_k}+\frac 1{2n_k}} \frac{1}{\left(1+\sum_{j=k+1}^{+\infty}c_j|z_j|^{-2n_k}\right)^{1/2n_k}}. $$

We can construct the sequence $\{\tilde r_m\}$ above setting $\tilde r_m=r_{n_k}$ for any $m$ for which it exists $k$ with $m=n_k$ and picking any other rational function with at most $m$ zeros and a $m$-order pole at $0$ for other values of $m$.
Now we use the assumptions \eqref{sequenceassumption} and properties of $\limsup$ to get
\begin{equation*}
\begin{split}
\limsup_m \left(\frac{\|r_m\|_K}{\|r_m\|_{L^2_{\mu_c}}}\right)^{1/m}\geq\limsup_k \left(\frac{\|r_{n_k}\|_K}{\|r_{n_k}\|_{L^2_{\mu_c}}}\right)^{1/n_k}\\
>\frac{1}{\liminf_k \left(1+\sum_{j=k+1}^\infty c_j|z_j|^{2 n_k}  \right)^{1/2n_k}}=1.
\end{split}
\end{equation*}

Thus $(K,\mu_c,\{0\})$ does not have the rational  sub-diagonal Bernstein Markov property, since the rational Bernstein Markov is a stronger property.
\item Lastly, the measure $d \mu:=d\mu_1+d\mu_2:=1/2\, ds|_{\partial \mathbb D}+1/2\,ds|_{1/2\partial \mathbb D}$ (here $ds$ denotes the standard arc length measure and $1/2\partial \mathbb D:=\{z:|z|=1/2\}$) has the rational Bernstein Markov property for $K=\partial \mathbb D\cup1/2 \partial \mathbb D$, $P=\{0\}$.

In order to show that, we pick any sequence of polynomials $\{p_k\}$ of degree not greater than $k$ and $\{m_k\}$ where $m_k\in\{0,1,\dots,k\}$, we consider the Bergman function for $\mu_1$ and $\mu_2$ and using \eqref{extremal} we get
\nalignstar{ &\left\|\frac{p_k}{z^{m_k}}\right\|_{L^2_\mu}=\|p_k\|_{L^2_{\mu_1}}+2^{m_k}\|p_k\|_{L^2_{\mu_2}}\geq\\
&(B_{k}^{\mu_1}(z_1))^{-1/2}|p_k(z_1)|\Big |_{z_1\in \partial\mathbb D}+2^{m_k} (B_{k}^{\mu_2}(z_2))^{-1/2}|p_k(z_2)|\Big|_{z_2\in 1/2\partial \mathbb D}=\\
&\left(\left(\frac {2\pi}{\sum_{j=0}^k|z_1^{j}|^2}\right)^{1/2}|p_k(z_1)|\right)_{z_1\in \partial\mathbb D}   +\;2^{m_k} \left(\frac {2\pi}{\sum_{j=0}^k2|{z_2}^{j}|^2}\right)^{1/2}p_k(z_2)|_{z_2\in \partial\mathbb D}.
}
Now we pick $z_1\in \partial\mathbb D$ and $z_2\in 1/2\partial\mathbb D$ maximizing $|p_k|$ and we get
\nalignstar{&\left\|\frac{p_k}{z^{m_k}}\right\|_{L^2_\mu}\geq \sqrt{\frac{2\pi} {k+1}}\|p_k\|_{\partial \mathbb D}\,+\,2^{m_k}\sqrt{\frac{3\pi}{4^{k+1}-1}}\|p_k\|_{1/2\partial \mathbb D}\geq\\
&\sqrt{\frac{3\pi}{4^{k+1}-1}}\cdot\left( \|p_k\|_{\partial \mathbb D}\,+\,2^{m_k} \|p_k\|_{1/2\partial \mathbb D}\right)=\\
&\sqrt{\frac{3\pi}{4^{k+1}-1}}\left( \left\|\frac{p_k}{z^{m_k}}\right\|_{\partial \mathbb D}\,+ \left\|\frac{p_k}{z^{m_k}}\right\|_{1/2\partial \mathbb D}\right)\geq\\
&\sqrt{\frac{3\pi}{4^{k+1}-1}} \left\|\frac{p_k}{z^{m_k}}\right\|_K.
}
It follows that, denoting $p_k/z^k$ by $r_k$, we have
$$\limsup_k \left(\frac{\|r_k\|_K}{\|r_k\|_{L^2_\mu}}\right)^{1/k}\leq \lim_k\left(\frac {4^{k+1}-1}{3\pi}  \right)^{1/(2k)}=1,$$
hence $(K,\mu,\{0\})$ has the rational Bernstein Markov property.
\end{enumerate}
\end{example}

The relation between these three properties is a little subtle: the examples above show that different aspects come in play from the geometry of $K$ and $P$ and the classes $\mathcal R(P),\mathcal Q(P)$. It will be clear later that the measure theoretic and potential theoretic features are important as well.

We relate the sub-diagonal rational Bernstein Markov property and the rational Bernstein Markov property to the weighted Bernstein Markov property with respect to a specific class of weights in Proposition \ref{rationaltoweighted} and \ref{rationaltoweighted2}; to do that we first recall the definition of weighted Bernstein Markov Property.
\begin{definition}[Weighted Bernstein Markov Property]\label{wbmpdegf}
Let $K\subset \C$ be a closed set and $w:K\rightarrow [0,+\infty[$ be an upper semicontinuous 
function, let $\mu\in \M^+(K)$, then the triple $[K,\mu,w]$ is said to have the 
weighted Bernstein Markov property if for any sequence of polynomials $p_k\in 
\wp^k$ we have
\begin{equation}
 \label{WBMP}
\limsup_k\left(\frac{\|p_kw^{k}\|_{K}}{\|p_kw^{k}\|_{L^2_\mu}}\right)^{1/k}\leq 
1.
\end{equation}
\end{definition}
In what follows we deal with weak$^*$ convergence of measures. We recall that, given a metric space $X$ and a Borel measure $\mu$ on $X$, the sequence of measures $(\mu_i)$ on $X$ is said to weak$^*$ converge to $\mu$ if for any bounded continuous function $f$ we have $\lim_i\left|\int_X f d\mu-\int_X fd\mu_i\right|=0;$ in such a case we write $\mu_i\wconverge{}\mu.$ Also, we recall that the space of Borel probability measures $\M^+_1(X)$ is weak$^*$ sequentially compact, that is for any sequence there exists a weak$^*$ converging subsequence.  

If $X$ is a compact space, then $\mathscr C(X)$ is a separable Banach space. It turns out that the space of Borel measures is isometrically isomorphic to the dual space $\mathscr C^*(X)$ and the topology of weak$^*$ convergence is generated by the family of semi-norms $\{p_f:\, f\in \mathcal F\}$ where $p_f(\mu):=|\int_X fd\mu|$ and $\mathcal F$ is any countable dense subset of $\mathscr C(X).$

Using these facts it is not difficult to prove the following statement that we will use in the proof of the next proposition.

\emph{Let $P$ be a compact set in $\C$ and $\sigma$ a Borel measure supported on it having total mass equal to $1$. There exists a sequence of arrays $\{(z_1^{(k)},\dots,z_k^{(k)})\}$ of points of $P$ such that we get}
\begin{equation}\label{wconvlemma}
\sigma_k:=\frac 1 k \sum_{j=1}^k\delta_{z_j^{(k)}}\wconverge{} \sigma.
\end{equation}

For any compact set  $P$ we introduce the following notation
\begin{align*}
\mathcal W(P)&:=\{e^{U^\sigma}: \sigma\in \M^+(P),0\leq \sigma(P)<\infty\}\;,\\
\mathcal W_1(P)&:=\{e^{U^\sigma}: \sigma\in \M_1^+(P)\},
\end{align*}
where $U^\sigma(z):=-\int \log|z-\z|d\sigma(\z)$ is the logarithmic potential of the 
measure $\sigma$ and we set by definition $U^0\equiv 0.$

\begin{proposition}\label{rationaltoweighted}
 Let $K\subset\C$ be a non polar compact set, $\mu\in \M^+(K)$ and $P$ any 
compact set disjoint by $K.$ Then the following are equivalent
\begin{enumerate}[(i)]
 \item $\forall w\in\mathcal W_1(P)$ the triple $[K,\mu,w]$ has the weighted Bernstein Markov 
Property.
\item $(K,\mu,P)$ has the sub-diagonal rational Bernstein Markov Property.
\end{enumerate} 
\end{proposition}
\begin{proof}[ of (i) implies (ii).]
Let us pick a sequence $\{r_k\}=\{ p_k / q_k \}$ in $\mathcal Q(P),$ where 
$q_k:=\prod_{j=1}^k(z-z_j),$ and let us set $\sigma_k:=\frac 1 {k}\sum_{j=1}^k\delta_{z_j}.$ Then we can notice that 
\begin{align*}
U^{\sigma_k}=&\int \log \frac 1 {|z-\z|}d\sigma_k(\z)=\frac 1 k\sum_{j=1}^k\log 
\frac{1}{|z-z_j|}=-\frac{1}{k}\log|q_k|. 
\end{align*}
 Thus, setting $U_k:=U^{\sigma_k},$ we have
\begin{equation}\label{extremalsequence}
a_k:=\left(\frac{\|r_k\|_K}{\|r_k\|_{L^2_{\mu}}}\right)^{1/k}=\left(\frac{
\|p_ke^{(kU_k)}\|_K }{\|p_ke^{(kU_k)}\|_{L^2_\mu} }\right)^{1/k}.
\end{equation}
Now we pick any maximizing subsequence $j\mapsto k_j$ for $a_k$, that is $\limsup_k 
a_k=\lim_j a_{k_j}.$ Let us pick any weak$^*$ limit $\sigma\in 
\M^+_1(P)$ and a subsequence $l\mapsto j_l$ such that 
$\tilde\sigma_l:=\sigma_{k_{j_l}}\wconverge{}\sigma.$ Moreover $\lim_l b_l:=\lim_l 
a_{k_{j_l}}=\limsup_k a_k.$

Let us notice that $U:=U^\sigma$ and all $U_l:=U^{\tilde\sigma_l}$ are harmonic functions on $\C\setminus P,$ moreover, due to \cite[Th. 6.9 I.6]{SaTo97}, $\{U_l\}$ converges quasi everywhere to $U.$ Notice that $U^{\tilde\sigma_l}:=-E\ast \tilde\sigma_l,$ where $E(z):=\log|z|$ is a locally absolutely continuous function on $\C\setminus \{0\}$, hence weak convergence of measures supported on $P$ implies local uniform convergence of potentials on $\C\setminus P.$

We can exploit this uniform convergence as follows. For any $\epsilon>0$ there exists $l_\epsilon$ such that for any $l>l_\epsilon$ we have
\begin{equation}\label{doubleestimate}
 U-\epsilon\leq U_l\leq U +\epsilon\;\;\text{uniformly on } K.
\end{equation}
Now we denote ${k_{j_l}}$ by $\tilde k_l$ and $ p_{\tilde k_l}$ by $\tilde p_l.$ It follows by 
\eqref{doubleestimate} that for $l$ large enough
\begin{align*}
 \|\tilde p_le^{\tilde k_l U_l}\|_K&\leq\|\tilde p_le^{\tilde k_l(U+\epsilon)}\|_K\leq 
e^{\tilde k_l\epsilon}\|\tilde p_le^{\tilde k_lU}\|_K,\\
\|\tilde p_le^{\tilde k_lU_l}\|_{L^2_\mu}&\geq\|\tilde 
p_le^{\tilde k_l(U-\epsilon)}\|_{L^2_\mu}\geq e^{-\epsilon\tilde k_l}\|\tilde 
p_le^{\tilde k_l U}\|_{L^2_\mu}  \;\text{ and thus}\\
\frac{\|\tilde p_le^{\tilde k_lU_l}\|_K}{\|\tilde p_le^{\tilde k_lU_l}\|_{L^2_\mu}}&\leq 
e^{2\tilde k_l\epsilon}\frac{\|\tilde p_le^{\tilde k_l U}\|_K}{\|\tilde 
p_le^{\tilde k_l U}\|_{L^2_\mu}}.
\end{align*}
Hence, exploiting $w:=e^{U}\in \mathcal W_1(P)$ and $\mu$ having the weighted Bernstein Markov property for such a weight, we have
\begin{align*}
 \limsup_k a_k&=\lim_l\left( \frac{\|\tilde p_le^{\tilde k_lU_l}\|_K}{\|\tilde p_le^{\tilde 
k_lU_l}\|_{L^2_\mu}} \right)^{1/\tilde k_l}\leq e^{2\epsilon} \lim_l\left( 
\frac{\|\tilde p_le^{\tilde k_lU}\|_K}{\|\tilde 
p_le^{\tilde k_l U}\|_{L^2_\mu}} \right)^{1/\tilde k_l}\\
&\leq e^{2\epsilon } \lim_l\left( 
\frac{\|\tilde p_lw^{\tilde k_l}\|_K}{\|\tilde 
p_lw^{\tilde k_l}\|_{L^2_\mu}} \right)^{1/\tilde k_l} = 
e^{2\epsilon }\longrightarrow 1\;\;\;\text{as }\epsilon\to 0.
\end{align*}\qed
\end{proof}
\begin{proof}[ of (ii) implies (i).] Suppose by contradiction that there exists $\sigma \in \mathcal W_1(P)$ such that $[K,\mu,\exp U^\sigma]$ does not have the weighted Bernstein Markov Property. 

We pick $\{z_1^{(k)},\dots,z_k^{(k)}\}_{k=1,\dots} $ and $\sigma_k= \frac 1 k \sum_{j=1}^k\delta_{z_j^{(k)}}$ as in \eqref{wconvlemma}.  

Let us set $w=\exp U^\sigma,$ $w_k=\exp U^{\sigma_k}.$ We can perform the same reasoning as above, using the absolute continuity of the $\log$ kernel away from $0$, to get $U^{\sigma_k}\to U^\sigma$ uniformly on $K.$ Thus for any $\epsilon>0$ we have $U^{\sigma_k}-\epsilon\leq U^{\sigma}\leq U^{\sigma_k}+\epsilon$ uniformly on $K$ for $k$ large enough. That is
\begin{equation}
 w_ke^{-\epsilon}\leq w\leq w_ke^{\epsilon}\;\;\text{ uniformly on 
$K$ for $k$ large enough.}\label{uniformestimate} 
\end{equation}

Notice that given any sequence $\{p_k\}$ such that $p_k\in \wp^k$ we have 
$$\{r_k\}:=\{p_k w_k^k\}=\left\{\frac{p_k}{\prod_{j=1}^k (z-z_j)} \right\}\in \mathcal Q(P).$$
Since we assumed that $[K,\mu,w]$ does not have the weighted Bernstein Markov property we can pick $p_k$ such that, using \eqref{uniformestimate},
\begin{align*}
1<& \limsup_k\left( \frac{\|p_kw^k\|_K}{\|p_kw^k\|_{L^2_\mu}}\right)^{1/k}
\leq& \limsup_k e^{2\epsilon}\left(\frac{\|p_kw_k^{k}\|_K}{\|p_kw_k^{k}\|_{L^2_\mu}}
\right)^ { 1/k }\\
\leq& e^{2\epsilon}\to 1\;\text{ as }\epsilon\to 0.
\end{align*}  This is a contradiction.\qed
\end{proof}

We can prove the following variant of the previous proposition by some minor modifications of the proof.

\begin{proposition}\label{rationaltoweighted2}
Let $K\subset\C$ be a non polar compact set, $\mu\in \M^+(K)$ and $P$ any 
compact set disjoint by $K.$ Then the following are equivalent
\begin{enumerate}[(i)]
 \item $\forall w\in\mathcal W(P)$ the triple $[K,\mu,w]$ has the weighted Bernstein Markov Property.
\item $(K,\mu,P)$ has the rational Bernstein Markov Property.
\end{enumerate} 
\end{proposition}
\begin{proof}[ of (i) implies (ii).]
We pick an extremal sequence in $\mathcal R(P)$ (i.e., for $a_k$ as in \eqref{extremalsequence}) $r_k:=\frac{p_{l_k}}{q_{m_k}}$, where $\degree p_{l_k}=l_k\leq k$ and $\degree q_{m_k}=m_k\leq k.$ 

We notice that
$$r_k=p_{l_k}e^{\left(m_k U^{\sigma_{m_k}} \right)}=p_{l_k}e^{\left(k U^{\frac{m_k}{k}\sigma_{m_k}} \right)}=: p_{l_k}e^{\left(k U^{\hat\sigma_{k}} \right)}\,,\;\text{where}$$
$\sigma_k$ are as in the previous proof.
Notice that the sequence of measures $\{\hat\sigma_{k}\}:=\{\frac{m_k}{k}\sigma_{m_k}\}$ has the property $\int_P d\hat\sigma_k\leq \int_P d\sigma_{m_k}=1$ since $m_k/k\leq 1$.

By the local sequential compactness we can extract a subsequence (relabeling indeces)  converging to any weak$^*$ closure point $\sigma$ that necessarily is a Borel measure such that $\int_P d\sigma\leq 1.$ Notice that $\sigma$ can be also the zero measure: here is the main difference between this case and Proposition \ref{rationaltoweighted} where each weak$^*$ limit has the same positive mass. 

Notice that $U^{\hat\sigma_{k}}$ converges to $U^{\sigma}$ uniformly on $K$ as in the previous proof, hence for any $\epsilon >0$ we can pick $k_\epsilon$ such that for any $k>k_\epsilon$ we have
$$ U^{\hat\sigma_{k}}-\epsilon\leq U^{\sigma}\leq U^{\hat\sigma_{k}}+\epsilon.$$
Therefore, seetting $w:=U^\sigma$ we have
\nalign{r_k e^{-k\epsilon}=p_{l_k}e^{\left(k U^{\hat\sigma_{k}}\right)}e^{(-k \epsilon)} \leq p_{l_k}e^{\left(k U^{\sigma} \right)}=p_{l_k}w^{k}\\ \leq p_{l_k}e^{\left(k U^{\hat\sigma_{k}}\right)}e^{(k \epsilon)}=r_k e^{k\epsilon}.}
The result follows by the same lines as in proof of Proposition \ref{rationaltoweighted}, using the weighted Bernstein Markov property of $[K,\mu,w]$ $\forall w\in\mathcal W(P).$ \qed
\end{proof}
\begin{proof}[of (ii) implies (i).]
Pick $\sigma$ such that $U^\sigma\in \mathcal W(P).$ If $\sigma=0$ we notice that the rational Bernstein Markov property is stronger than the usual Bernstein Markov property.

If $\sigma$ is not the zero measure we set $c:=\int_P d\sigma$, $\hat\sigma=\sigma/ c\in \mathcal M_1^+(P)$, and we pick a sequence of natural numbers $0\leq m_k\leq k$ such that $\lim_k m_k/k=c.$ We find $\sigma_k\in \mathcal M_1^+(P)$, $\sigma_k:=(1/m_k)\sum_{j=1}^{m_k}\delta_{z_j^{(m_k)}}$ such that $\sigma_k\to^* \hat\sigma$ as in the previous proof, thus $\frac{m_k}{k}\sigma_k\to^*\sigma.$

It follows that 
\begin{equation}
m_kU^{\sigma_k}+k\epsilon=k(\frac{m_k}{k} U^{\sigma_k}-\epsilon)\leq kU^{\sigma}\leq k(\frac{m_k}{k} U^{\sigma_k}-\epsilon)=m_kU^{\sigma_k}-k\epsilon,\label{uniformestimate2}
\end{equation}
for $k$ large enough.

We can work by contradiction supposing that $[K,\mu,U^{\sigma}]$ does not satisfy the weighted Bernstein Markov property and following the same lines of the proof of (\emph{ii}) implies (\emph{i}) of the previous proposition using \eqref{uniformestimate2} instead of \eqref{uniformestimate}. 
 \qed
\end{proof}

\begin{remark}
The combination of the two previous propositions proves in particular that if $(K,\mu,P)$ has the sub-diagonal rational Bernstein Markov property and $(K,\mu)$ has the Bernstein Markov property, it follows that $(K,\mu,P)$ has the rational Bernstein Markov property.

On the other hand if $(K,\mu,P)$ has the sub-diagonal rational Bernstein Markov property but not the rational Bernstein Markov property, it follows that $(K,\mu)$ does not satisfy the Bernstein Markov property.
\end{remark} 

According to Proposition \ref{rationaltoweighted2}, our original question boils down 
to \emph{whether the Bernstein Markov property implies the weighted Bernstein Markov property for any weight in the class $\mathcal W(P)$}. In the next theorem we give two possible sufficient conditions for that, corresponding to two different situations that are rather extremal in a sense. The reader is invited to compare them with  situation of Example 1(a) and 1(b).

We denote by $S_K$ the \emph{Shilov boundary} of $K$ with respect to the uniform algebra $\mathcal P(K)$ of functions that are uniform limits on $K$ of entire functions (or equivalently polynomials). We recall that $S_K$ is defined as the smallest closed subset $B$ of $K$ such that $\max_{z\in K}|f(z)|=\max_{z\in B}|f(z)|$ for all $f\in \mathcal P(K).$

We use the standard notation for the \emph{polynomial hull} of a compact set $K$, that is
\begin{equation}
\label{polyhull}
\hat K:=\{z\in \C:\, |p(z)|\leq \|p\|_K\,,\,\forall p\in \wp\},
\end{equation}  
where $\wp:=\cup_{k\in \N}\wp^k.$
\begin{theorem}\label{mainresultsec2}
 Let $K\subset\C$ be a compact non polar set and $\mu\in \M^+(K)$ be such that 
$\support \mu=K$ and $(K,\mu)$ has the Bernstein Markov Property. For a 
compact set $P\subset \C$ such that $K\cap P=\emptyset$, suppose that one of the following 
occurs.
\begin{flalign*}
&\text{Case a:}&S_K=K.&&\\ 
&\text{Case b:}&\hat K\cap P=\emptyset.&& 
\end{flalign*}
Then the triple $[K,\mu,w]$ has the weighted Bernstein Markov Property with respect to
any weight $w\in \mathcal W(P)$ and thus $(K,\mu,P)$ has the rational 
Bernstein Markov Property. 
\end{theorem}
\begin{proof}
Let us pick $\sigma\in \M^+(P)$ and set $w=\exp U^\sigma,$ also we pick a sequence $\{p_k\}$, 
where $p_k\in \wp^k.$
We show that in both cases $[K,\mu,w]$ has the weighted Bernstein Markov Property with respect to any weight $w\in \mathcal W_1(P),$ the rest following by Proposition \ref{rationaltoweighted2}.

\textbf{Case a.} We first recall (see \cite[Lemma 3.2.4 pg. 70]{StaTo92}) that the set 
$\{|g|:g\in \wp\}$ is dense in the cone of positive continuous functions on $S_K,$ which $w$ belongs to. 

For any $\epsilon>0$ we can pick $g_\epsilon\in \wp^{m_\epsilon}$ such that
\begin{equation}
 (1-\epsilon)|g_\epsilon|\leq w\leq (1+\epsilon)|g_\epsilon|.\label{twosideapprox}
\end{equation}
Notice that $|g_\epsilon|^k=|g_\epsilon^k|=|\tau_{\epsilon,k}|,$ where $\tau_{\epsilon,k}\in \wp^{m_\epsilon k}.$

If for any $p_k\in \wp^k$ we set $\tilde p_k:=\tau_{\epsilon,k}p_k\in \wp^{(m_\epsilon+1)k}$, then we have
\nalign{
 \|p_k w^k\|_K\leq & (1+\epsilon)^k\|\tau_{\epsilon,k}p_k\|_K=\|\tilde p_k\|_K,\\
 \|p_k w^k\|_{L^2_\mu}\geq & (1-\epsilon)^k\|\tau_{\epsilon,k}p_k\|_{L^2_\mu}=\|\tilde 
p_k\|_{L^2_\mu},\text{ and thus}\\
 \left(\frac{\|p_k w^k\|_K}{\|p_k 
w^k\|_{L^2_\mu}}\right)^{1/k}\leq&\frac{1+\epsilon}{1-\epsilon} 
\left[\left(\frac{\|\tilde p_k\|_K}{\|\tilde 
p_k\|_{L^2_\mu}}\right)^{\frac 1{(m_\epsilon+1)k}}\right]^{m_\epsilon+1}.\label{caseaestimate}
}
Using the polynomial Bernstein Markov property of $(K,\mu)$ and the arbitrariness of $\epsilon>0$ we can conclude 
that $\limsup_k  \left(\frac{\|p_k w^k\|_K}{\|p_k 
w^k\|_{L^2_\mu}}\right)^{1/k}\leq 1.$

\textbf{Case b.} Suppose first that $\hat K$ is connected, then it follows that  there exists an open neighbourhood $D$ of $\hat K$ which is a simply connected domain and $P\cap D=\emptyset.$ We recall that any harmonic function on a simply connected domain is the real part of a holomorphic one. Hence, being $U^\sigma$ harmonic on $D,$  we can pick $f$ holomorphic on $D$ such that 
\begin{equation}
w=\exp U^\sigma=\exp \Re f=|\exp f|.\label{harmonicexponential}
\end{equation}

Since $g:=\exp f$ is an holomorphic function on $D$, by Runge Theorem, we can uniformly approximate it by polynomials $g_\epsilon$ on $\hat K:=\{z\in \C, |p(z)|\leq \|p\|_K\,\forall p\in \mathscr P(\C)\}$. Now we can conclude the proof by the same argument \eqref{caseaestimate} and \eqref{harmonicexponential} of the \emph{Case a} above.

If otherwise $\hat K$ is not known to be connected, we apply the following version of the Hilbert Lemniscate Theorem \cite[Th. 16.5.6]{Hi}, given any open neighbourhood $U$ of $\hat K$ not intersecting $P$ we can pick a polynomial $s\in \wp$ such that $|s(z)|> \|s\|_{\hat K}=\|s\|_K$ for any $z\in \C\setminus U.$   

It follows that, picking a suitable positive $\delta$, the set $E:=\{|s|\leq\|s\|_K+\delta\}$ is a closed neighbourhood of $\hat K$ not intersecting $P.$

Notice that the set $E$ has at most $\degree s$ connected components $E_j$ and by definition it is polynomially convex. Moreover the Maximum Modulus Theorem implies that each $D_j:=\interior E_j$ is simply connected or the disjoint union of a finite number of simply connected domains that we do not relabel.

For any $j=1,2,\dots,\degree s$ we set $w_j:=w|_{D_j}$. We can find holomorphic functions $f_j$ and $g_j$ on $D_j$, continuous up to its boundary, such that $w_j=|\exp f_j|=|g_j|.$

Now notice that the function $g(z)=g_j(z)$ $\forall z\in D_j$ is holomorphic on $D$ and continuous on $E$, since $D$ is the disjoint union of the sets $D_j$'s. Hence we can apply the Mergelyan Theorem to find for any $\epsilon>0$ a  polynomial $g_\epsilon$ such that 
$$(1-\epsilon)|g_\epsilon(z)|\leq w(z)\leq (1+\epsilon)|g_\epsilon(z)|\;\;\forall z\in E\supseteq K. $$
We are back to the \emph{Case a} and the proof can be concluded by the same lines.\qed
\end{proof}

\section{A sufficient mass-density condition for the rational Bernstein Markov property}

In the case of $K=\support \mu$ being a regular set for the Dirichlet problem, the Bernstein Markov Property for $(K,\mu)$  is equivalent (cfr. \cite[Th. 3.4]{B97}) to $\mu\in \Reg$. A positive Borel measure is in the class $\Reg$ or \emph{has regular n-th root asymptotic behaviour} if for any sequence of polynomials $\{p_k\}$ one has
\begin{equation}
\limsup_k\left(\frac{|p_k(z)|}{\|p_k\|_{L^2_\mu}}\right)^{1/\degree p_k}\leq 1\; \text{ for } z\in K\setminus N,\;N\subset K, N\text{ is polar.}\label{regdef}
\end{equation}
However, the definition can be given in terms of other equivalent conditions, see \cite[Th. 3.1.1, Def. 3.1.2]{StaTo92}. We recall for the reader's convenience that a set $P$ is \emph{polar} if it is locally representable as a subset of the $\{-\infty\}$ level set of a subharmonic function.

Moreover in \cite[Th. 4.2.3]{StaTo92} it has been proven that any Borel compactly supported finite measure having regular support $K\subset \C$ and enjoying a \emph{mass density condition} ($\Lambda^*$-criterion \cite[pag. 132]{StaTo92}) is in the class $\Reg$, consequently $(K,\mu)$ has the Bernstein Markov property. In order to fulfil such $\Lambda^*$ condition a measure needs (roughly speaking) to be thick in a measure-theoretic sense on a subset of its support which has full logarithmic capacity (see equation \eqref{massdensity} below for the rigorous statement).

Notice that, even if this $\Lambda^*$ criterion is not known to be necessary for the Bernstein Markov property, in \cite{StaTo92} authors show that the criterion has a kind of sharpness property and no counterexamples to the conjecture of $\Lambda^*$ being necessary for the Bernstein Markov property are known. 
 Moreover, this mass density sufficient condition has been extended (here the logarithmic capacity has been substituted by the relative Monge-Ampere capacity with respect to a ball containing the set $K$) to the case of several complex variables by Bloom and Levenberg \cite{BlLe99}.

Here we observe that under the hypothesis of Theorem \ref{mainresultsec2} this condition turns out to be sufficient for the rational Bernstein Markov property as well; we state this in Theorem \ref{PCcase} then we generalize this result in Theorem \ref{NPCcase}.

We recall the definition of the \emph{logarithmic capacity} $\capa(\cdot)$ of a compact subset of the complex plane
\begin{equation}\label{logcap}
\capa(K):=\sup_{\mu\in \mathcal M^+_1(K)}\exp\left( -I[\mu] \right),
\end{equation}
where we denote by 
\begin{equation}
I[\mu]:=\int U^\mu d\mu =\int\int \log\frac 1 
{|z-\z|}d\mu(z)d\mu(\z)\label{logenergy} 
\end{equation}
the logarithmic energy of the measure $\mu.$

The existence of a minimizers for $I[\cdot]$ holds true provided $K$ is a non polar set \cite[Part I]{SaTo97} while the uniqueness follows by the strict convexity of $I[\cdot]$. The unique minimizer is named \emph{equilibrium measure} or \emph{extremal measure} and is denoted by $\mu_K$. It is a fundamental result that, for non-polar $K$,
\begin{equation}\label{equilibriummeasure}
\mu_K=\Delta g_K(z,\infty)\;\;\text{ and }\;\;g_K(z,\infty)=\int\log|z-\z|d\mu_K(\z)-\log\capa(K),
\end{equation}
where $g_K(z,\infty)$ is the Green function for the unbounded component $\Omega_K$ of $\C\setminus K$ with logarithmic pole at $\infty.$ Here the Laplacian has to be intended in the sense of distributions and has been normalized to get a probability measure.

\begin{theorem}\label{PCcase}
Let $K\subset\C$ be a compact regular set and $P\subset \Omega_K$ be compact. Let $\mu\in \M^+(K)$, $\support \mu=K$ and suppose that there exists $t>0$ such that
\begin{equation}\label{massdensity}
\lim_{r\to 0^+}\capa\left(\{z\in K:\mu(B(z,r))\geq r^t\} 
\right)=\capa(K).
\end{equation}
Then $(K,\mu,P)$ has the rational Bernstein Markov Property.  
\end{theorem}
\begin{proof}[\textsc i]
By \cite[Th. 4.2.3]{StaTo92} it follows that $(K,\mu)$ has the Bernstein Markov property, by Theorem \ref{mainresultsec2} \emph{Case b} we can conclude that the rational Bernstein Markov property holds for $(K,\mu,P)$ for any $P\subset\Omega_K$ as well.\qed
\end{proof}
\begin{proof}[\textsc {ii}] 
See subsection \ref{anotherproof}.
\end{proof}

If we remove the hypothesis $P\subset \Omega_K$, then Theorem \ref{mainresultsec2} is no more applicable. We go around such a difficulty in the case $K\subset \Omega_P$ by a suitable conformal mapping $f$ of a neighbourhood of $K\cup P$ given by the Proposition \ref{techlemma} below.    

We recall, for the reader's convenience, the definitions of \emph{Fekete points} and \emph{transfinite diameter}. Given any compact set $K$ in the complex plane, for any positive integer $k$, a set of Fekete points of order $k$ is an array $\bs{z_k}=\{z_0,\dots,z_k\}\in K^k$ that maximizes the product of distances of its points among all such arrays, that is
\begin{equation*}
V_k(\bs{z_k}):=\prod_{1\leq i<j\leq k}|z_i-z_j|=\max_{\bs{\z}\in K^k} \prod_{1\leq i<j\leq k}|\z_i-\z_j|.
\end{equation*}
Notice that such maximizing array does not need to be unique.

It turns out that, denoting by $\delta_k(K):=\left(\max_{\bs{\z}\in K^k} V_k(\bs{\z})\right)^{\frac 2{k(k+1)}}$ the $k$-th diameter of $K,$ we have
\begin{equation}\label{pottheo}
\lim_k \delta_k(K)=:\delta(K)=\capa(K),
\end{equation}
where $\delta(K)$ is the \emph{transfinite diameter} of $K$ (existence of the limit being part of the statement). We refer the reader to \cite{Rans,SaTo97,Sa10} for further details. 

Recall that we indicate by $\hat E$ the polynomial hull of the set $E$, see \eqref{polyhull}.
\begin{proposition}\label{techlemma}
 Let $K,P\subset\C$ be compact sets, where $K\cap \hat P=\emptyset$.
Then there exist $w_1,w_2,\dots,w_m\in \C\setminus (K\cup \hat P)$ and $R_2>R_1>0$ such that denoting by $f$ the function 
$z\mapsto\frac 1 {\prod_{j=1}^m(z-w_j)}$ we have
 \begin{align*}
 K&\subset\subset \{|f|<R_1\},\\
 P&\subset\subset \{R_1<|f|<R_2\}.
\end{align*}
\end{proposition}
\begin{proof}We first suppose $P$ to be not polar.

Moreover we show that we can suppose without loss of generality that
\begin{equation}\label{capcondition}
 \log\delta( P)<\min_K g_{P}(\cdot,\infty).
\end{equation}
To do that, consider $0<\lambda<\frac 1{\delta( P)}$ and notice that
$$\log\delta(\lambda P)=\log \lambda\delta (P)<0.$$
On the other hand one has $g_{\lambda P}(z,\infty)=g_{P}(\frac z \lambda,\infty)$, thus it follows that
$$\min_{z\in K} g_P(z,\infty)=\min_{z\in \lambda K} g_{\lambda P}(z,\infty)>0>\log\delta(\lambda P),$$
where the first inequality is due to the assumption $K\cap \hat P=\emptyset.$

If we build $\tilde f$ as in the proposition for the sets $P':=\lambda P$ and $K':=\lambda K$, then $f:=\tilde f \circ \frac 1 \lambda$ enjoys the right properties for the original sets $P,K.$ Hence in the following we can suppose \eqref{capcondition} to hold.

Let us pick $0<\rho<\bar\rho:=d(\hat P,K)/2$, where $d(A,B):=\inf_{x\in A,y\in B}|x-y|$, and consider the set $\hat P^\rho$.

For the sake of an easier notation we denote by $g(z)$ and $g_\rho(z)$ the functions $g_{P}(z,\infty)$ and $g_{\hat P^\rho}(z,\infty)$.

For any $k\in \N$ let us pick any set $Z_k(\rho):=\{z_1^{(k)},\dots,z_{k}^{(k)}\}$ of Fekete points for $\hat P^\rho,$ moreover we denote the polynomial $\prod_{j=1}^k(z-z_j^{(k)})$ by $q_k.$ Notice that $Z_k(\rho)\subset (\partial \hat P^\rho)^k\subset(\C\setminus(K\cup P))^k$, hence $\{z_1^{(k)},\dots,z_{k}^{(k)}\}$ is an admissible tentative choice for $w_1,w_2,\dots,w_k.$

Let us set
\begin{align*}
a(\rho):=&\min_K g_\rho,\\
a:=&\min_{\rho\in[0,\bar\rho]}a(\rho)=a(\bar\rho),\\
b:=&\max_{\rho\in [0,\bar\rho]}\max_K g_\rho=\max_Kg.
\end{align*}

We recall that (see \cite[III Th. 1.8]{SaTo97})
$$\lim_k \frac{1}{k}\log^+|q_k|=g_\rho,\;\text{ locally uniformly on } \C\setminus \hat P^\rho.$$ 
Thus for any $\epsilon>0$ we can choose $m(\epsilon)\in \N$ such that
$$\left\|\frac{1}{m}\log^+|q_m|-g\right\|_{B(\rho)}< \epsilon\;\;\forall m\geq m(\epsilon),$$
where $B(\rho):=\{z\in \C:a\leq g_\rho(z)\leq b\},$ notice that $\hat P^\rho\cap B(\rho)=\emptyset.$

Then, taking $\epsilon<a$ we have $\forall m\geq m(\epsilon)$
\begin{equation}\label{Dr}\begin{split}
K\subset \left\{a(\rho)-\epsilon\leq\frac{1}{m}\log^+|q_m|\leq b+\epsilon 
\right\}=\\
\left\{ e^{m(a(\rho)-\epsilon)} \leq|q_m| \leq e^{m(b+\epsilon)} \right\}=:A(\epsilon,\rho,m).
\end{split}\end{equation}
On the other hand, exploiting the extremal property of Fekete polynomials \cite[Th. 5.5.4 (b)]{Rans}, we have $\|q_m\|_{\hat P^\rho}\leq \delta_m(\hat P^\rho)^m,$
where $\delta_m(E)$ is the $m$-th order diameter of $E.$ In other words
$$P\subset \left\{|q_m|\leq \delta_m(\hat P^\rho)^m \right\}=:D(\rho,m).$$

In order to prove that $A(\epsilon,\rho,m)\cap D(\rho,m)=\emptyset,$ for suitable $\epsilon >0$, $\dist(K,\hat P)>\rho>0$ and $m>m(\epsilon)$, we need to show that for such values of parameters 
\begin{equation}\label{mcondition}
\log\delta_m(\hat P^\rho)<a(\rho)-\epsilon.
\end{equation}

In such a case the function $f(z):=\frac{1}{q_m(z)}$ satisfies the properties of the proposition since
$$\|f\|_K\leq e^{(-m(a(\rho)-\epsilon))}<\delta_m(\hat P^\rho)^{-m}\leq \min_{P}|f|.$$ 

To conclude, we are left to prove that we can choose admissible $m,$ $\rho>0$ and $\epsilon>0$  such that \eqref{mcondition} holds. To do that we recall that, since $P=\cap_{l\in \N} P_{\frac 1 l}$ , by \cite[Th. 5.1.3]{Rans} we have
$$\delta(P)=\lim_l \delta(P_\frac 1 l)=\lim_l\lim_m\delta_m(P_\frac 1 l).$$
By the same reason $g_{1/m}$ is uniformly converging by the Dini's Lemma to $g$ on a neighbourhood of $K$ not intersecting $P_{\bar\rho}$.

Therefore, it follows by \eqref{Dr} and \eqref{capcondition} that possibly shrinking $\epsilon$ to get $$0<\epsilon<\min\{a,\min_K g -\log\delta(P)\}\;\; \text{ we have}$$
$$\lim_l\lim_m\log\delta_m(P_\frac 1 l)=\log\delta(P)\,<\,\min_K g-\epsilon=\lim_m \min_K g_{1/m}-\epsilon.$$
Hence (possibly taking $\epsilon'<\epsilon$) there exists a increasing subsequence $k\mapsto l_k$ with 
$$\lim_m\log\delta_{m}(P_{1/l_k})< \lim_m\min_K g_{1/m}-\epsilon' \text{ for any }\,k\in \N.$$
In the same way we can pick a subsequence $k\to m_k$ such that $\log\delta_{m_k}(P_{1/l_k})< \min_K g_{1/m_k}-\epsilon''$ for all $k\in \N$. Taking $k$ large enough to get $m_k>m(\epsilon'')$ and setting $m:=m_k$, $\rho:=1/l_k$  suffices.

In the case of $P$ being a polar subset of $\C$ we observe that for any positive $\rho$ the set $\hat P^\rho$ is not polar since it contains at least one disk. Moreover notice that $\lim_m\delta_m(P_{1/m})=\log\delta(P)=-\infty$ whereas the sequence of harmonic (on a fixed suitable neighbourhood of $K$) functions $g_{1/m}$ is positive and increasing. Equation \eqref{Dr} is then satisfied for $m$ large enough. The rest of the proof is identical.\qed 
\end{proof}

We use the standard notation $f_*\mu(A):=\int_{f^{-1}(A)}d\mu$ for any Borel set $A\subset \C$.

If we use Proposition \ref{techlemma} and set $E:=f(K),\,Q:=f(P)$ we can see that $\widehat{E}\cap Q=\emptyset$ thus $E,Q$ are precisely in the same relative position as in the Theorem \ref{PCcase}. Therefore we are now ready to state a sufficient condition for the rational Bernstein Markov property under more general hypothesis, where we do not assume $\hat K\cap P=\emptyset$.

\begin{theorem}[Mass-Density Sufficient Condition]\label{NPCcase}
 Let $K,P\subset \C$ be compact disjoint sets where $K$ is regular with respect to the Dirichlet problem and $\hat P\cap K=\emptyset$. Let $\mu\in \M^+(K)$ be such that $\support 
\mu=K$ and suppose that there exist $t>0$ and $f$ as in Proposition \ref{techlemma} 
such that the following holds
\begin{equation}
\lim_{r\to 0^+}\capa\left(\{z\in f(K):f_*\mu(B(z,r))\geq r^t\} 
\right)=\capa(f(K)).\label{massdensity2}
\end{equation}
Then $(K,\mu,P)$ has the rational Bernstein Markov Property.
\end{theorem}
\begin{proof}By Theorem \ref{PCcase} it follows that the triple $(E,f_*\mu,Q)$ has the 
rational Bernstein Markov Property. 

To conclude the proof it is sufficient to notice that for any sequence $\{r_k\}$ in 
$\mathcal R(P)$, the sequence $\{\tilde r_j\}$ defined by
\begin{equation*}
 \tilde r_j:=r_{\lfloor j/m\rfloor}\circ f\;\;j=1,2,\dots
\end{equation*}
is an element of $\mathcal R(Q).$
Moreover by the rational Bernstein Markov property of $(E,f_*\mu,Q)$ we can pick $c_j>0$ such that $\limsup_j c_j^{1/j}\leq 1$
and 
$$\|r_k\|_K=\|\tilde r_{mk}\|_E\leq c_{mk}\|\tilde r_{mk}\|_{L^2(f_*\mu)}\leq 
c_{mk}\|r_k\|_{L^2(\mu)}.$$
Thus we have
$$\left(\frac{\|r_k\|_K}{\|r_k\|_{L^2(\mu)}}\right)^{1/k}\leq 
\left(c_{mk}^{1/(mk)}\right)^m\to 1^m=1.$$\qed
\end{proof}
We can also state the above result in a simpler way, thought not completely equivalent.
\begin{corollary}\label{easierstatement}
Let $K,P\subset \C$ be compact sets where $K$ is regular with respect to the Dirichlet problem and $\hat P\cap K=\emptyset$. Let $\mu\in \M^+(K)$ be such that $\support 
\mu=K$ and suppose that there exist $t>0$ and $f$ as in Proposition \ref{techlemma} such that the following holds
\begin{equation}
\lim_{r\to 0^+}\capa\left(f\left(\{\z\in K:\mu(B(\z,r))\geq r^t\} 
\right)\right)=\capa(f(K)).\label{massdensity3}
\end{equation}
Then $(K,\mu,P)$ has the rational Bernstein Markov Property.
\end{corollary}
\begin{proof}
Let $L:=\Lip_K f=\inf\{L: |f(x)-f(y)|<L|x-y|, \text{ for all }x,y\in K\}$, we set
\begin{align*}
A_r&:=\{\z\in K:\, \mu(B(\z,r/L))\geq r^t\}\\
D_r&:=\{z\in f(K):\,f_*\mu(B(z,r))\geq r^t\}.
\end{align*}
We observe that if $\z_0\in A_r$ then $z_0:=f(\z_0)$ lies in $D_r$. For, notice that
$$f_*\mu(B(z_0,r))=\int_{f^{-1 }(B(z_0,r))}d\mu\geq\int_{B(\z_0,r/L)}d\mu$$
since $f(B(\z_0,r/L))\subseteq B(z_0,r).$ Therefore $f(A_r)\subseteq D_r.$

If we suppose that $\capa(f(A_r))\to\capa(f(K)),$ then it follows that $\capa(D_r)\to \capa(f(K))$ as well by the inequality $\capa(f(K))\geq \capa(D_r)\geq \capa(f(A_r))\to \capa(f(K)).$

Now consider the set $B_r:=\{\z\in K:\, \mu(B(\z,r))\geq r^{t'}\},$ for some $t'>t$, condition \eqref{massdensity3} says $\lim_{s\to 0^+}\capa(f(B_s))=\capa(f(K))$. Now take $s=r/L$ and notice that for small $r$ we have $\left(\frac r L\right)^{t'}\geq r^t$, thus by condition \eqref{massdensity3} it follows that $\lim_{r\to 0^+}\capa(f(A_r))=\capa(f(K)).$ By the previous argument condition \eqref{massdensity2} follows and Theorem \ref{NPCcase} applies.\qed
\end{proof}

\begin{example} We go back to the case of the Example 1 (e) to show that the same conclusion follows by applying Corollary 1. Let us recall the notation. We consider the annulus $A:=\{z:1/2\leq |z|\leq 1\}$, set $K:=\partial A$, $P:=\{0\}$ and $\mu:=1/2 ds|_{\partial \mathbb D}+1/2 ds|_{\frac 1 2\partial \mathbb D}$, where $ds$ is the standard arc length measure.

We proceed as in Proposition \ref{techlemma} to build the map $f$: we take $\rho=0.1$ and for each $m\in \N$ we pick a set of Fekete points for $P^\rho=\{|z|\leq 0.1\}$.

In this easy example $m=2$ suffices to our aim, so we can choose $w_1=0.1$, $w_2=-0.1$, $f(z)=\frac 1{(z-w_1)(z-w_2)}=\frac{1}{z^2-0.01}.$

We notice that $f$ is a holomorphic map of a neighbourhood $K^\delta$ of $K$ and we can compute its Lipschitz constant $\Lip_{K}(f):=\inf\{L>0:|f(x)-f(y)|\leq L|x-y|,\forall x\neq y\in K\}
$ as follows.
$$L_\delta:=\Lip_{K^\delta}(f)=\| f'\|_{K^\delta}=\max_{z\in K^\delta}\left|\frac{-2z}{(z^2-0.01)^2}  \right| .$$
For instance, taking $\delta= 0.1$ we get $L_\delta=\frac{4(1-2\delta)}{1-4\delta}=5.\overline 3.$

For any $\z\in \partial \mathbb D$ and $r<1/2$ we have
\nalignstar{
&\mu\left (B(\z,r)\right)=\frac 1 2 \int_{B(\z,r)\cap \partial \mathbb D}ds=\frac 1 2\int_{\arg(\z)-\arcsin\left(r\right)}^{\arg(\z) +\arcsin\left(r\right)}1\,d\theta\\
&=\arcsin\left(r\right),
}
similarly for any $\z\in 1/2\partial \mathbb D$ we have
\nalignstar{
&\mu\left (B(\z,r)\right)=\frac 1 2 \int_{B(\z,r)\cap 1/2\partial \mathbb D}ds=\frac 1 2\int_{\arg \z-2\arcsin\left(r\right)}^{\arg \z +2\arcsin\left(r\right)}1\,\frac{d\theta}{2}\\
&=\arcsin(r).
}
Notice that taking $t=1$ and $r<1/2$ \eqref{massdensity3} is satisfied since $\{\z\in K:\mu(B(\z,r))\geq r\}=K$ for all $0<r<1/2.$

Finally we notice that also $(A,\mu,P)$ has the rational Bernstein Markov property (as we observed in Example 1 (e)) since any rational function having poles on $P$ achieves the maximum of its modulus on $K$.
\end{example}
It is worth to notice that a measure $\mu$ can satisfy \eqref{massdensity3} even if the mass of balls of radius $r$ decays very fast (e.g. faster than any power of $r$) as $r\to 0$ at some points of the support of $\mu$. This is the case of the following example.

\begin{example}
Let us consider the measure $\mu$, where 
$$\frac{d\mu}{d\theta}:=\exp\left(\frac{-1}{1-\left(\frac{\theta}{\pi}\right)^2}\right)\;,\;-\pi\leq \theta\leq \pi$$
defined on the unit circle $\partial\mathbb D$ and pick as pole set $P:=\{0\}.$

\begin{align}
&\mu(B(e^{i\theta},r))= \int_{\theta - 2\arcsin r/2}^{\theta +2\arcsin r/2} \exp\left( \frac{-\pi^2}{\pi^2-u^2}\right) du   \\
&\geq 
\begin{cases}
4\arcsin r/2 \exp\left( \frac{-\pi^2}{\pi^2-\left( \theta+2\arcsin r/2\right)^2} \right)&, 0\leq\theta<\pi-2\arcsin r/2\\
4\arcsin r/2\exp\left( \frac{-\pi^2}{\pi^2-(\theta - 2\arcsin r/2)^2} \right)&, -\pi+2\arcsin r/2\leq\theta\leq 0.
\end{cases}.\label{massestimate}
\end{align}
We try to test condition \eqref{massdensity3} using $t=1$ and the map $f(z):=\frac{1}{z-0.01}$ which is a bi-holomorphism of a neighbourhood of $\partial \mathbb D$. Therefore the condition $\lim_{r\to 0^+}\capa(f(K_r))=\capa(f(K))$ of Corollary 1 for sets $K_r\subseteq K$ is equivalent to $\lim_{r\to 0^+}\capa K_r=\capa K$ and we are reduce to test the simpler condition
\begin{equation}
\lim_{r\to 0^+}\capa\left(\{z\in \partial \mathbb  D:\mu(B(z,r))\geq r\}\right)=:\lim_{r\to 0^+}\capa K_r=\capa(\partial \mathbb D).\label{simpler}
\end{equation}
It is not difficult to see by \eqref{massestimate} that 
\begin{align*}
&K_r\supset\\
&\left\{e^{i\theta}: \theta\in[0,\pi-2\arcsin r/2[\,, \exp\left( \frac{-\pi^2}{\pi^2-\left( \theta+2\arcsin r/2\right)^2}\right)\geq \frac{r}{4\arcsin r/2}\right\}\bigcup\\
&\;\;\left\{e^{i\theta}: \theta\in]-\pi+2\arcsin r/2,0]\,, \exp\left( \frac{-\pi^2}{\pi^2-\left( \theta-2\arcsin r/2\right)^2}\right)\geq \frac{r}{4\arcsin r/2}\right\}=\\
&K_r^1\cup K_r^2=:\tilde K_r,
\end{align*}
where
\begin{align*}
K_r^i&=\{e^{i\theta},\theta\in[a_i,b_i]\}\\
a_1&=\max\left\{0 ,2\arcsin r/2-\pi\sqrt{1+\frac{1}{\log\frac{4\arcsin r/2}{2}}}\right\}\\
b_1&=\min\left\{\pi-2\arcsin r/2,\pi\sqrt{1-\frac 1{\log{\frac{2\arcsin r}{r}}}}-\arcsin r \right\}\\
a_2&=\min\left\{-\pi+2\arcsin r/2,\pi\sqrt{1-\frac 1{\log{\frac{2\arcsin r}{r}}}}-\arcsin r \right\}  \\
b_2&=\max\left\{0 ,-2\arcsin r/2+\pi\sqrt{1+\frac{1}{\log\frac{4\arcsin r/2}{2}}}\right\}.
\end{align*}
It is not difficult to see that for $r\to 0^+$ we have $[a_1,b_1]=[0,\pi-2\arcsin r/2]$, $[a_2,b_2]=[-\pi+2\arcsin r/2,0]$, hence $K_r\supseteq \{e^{i\theta},\theta\in [-\pi+2\arcsin r/2,\pi-2\arcsin r/2]$

We recall that the logarithmic capacity of an arc of circle of radius $1$ and length $\alpha$ is $\sin(\alpha/4)$; see \cite[pg. 135]{Rans}. Therefore we have
\nalign{
&\capa(\partial \mathbb D)\geq\lim_{r\to 0^+}\capa (K_r)\geq \lim_{r\to 0^+}\capa (\tilde K_r)=\\
&\lim_{r\to 0^+}\sin\left(\frac{2\pi-4\arcsin r/2}{4}\right)=1=\capa(\partial \mathbb D),
}
this proves \eqref{simpler} and since we considered a bi-holomorphic map $f$ \eqref{massdensity3} follows. By Corollary \ref{easierstatement} we can conclude that $\{\partial \mathbb D, \mu,\{0\}\}$ has the rational Bernstein Markov property. 
\end{example}

\section{Convergence of Green functions and mass density condition}
The aim of this section is to relate the convergence of logarithmic capacities of compact subsets $K_j$ of a given compact regular set $K$ to the uniform convergence of the Green functions $g_{K_j}(z,a)$ to $g_K(z,a)$ with poles $a$ in a given compact set $P$ disjoint by $K$. We provide a one variable version (see Th. \ref{approximationtheorem} below) of \cite[Th. 1.2]{BlLe99} adapted to our setting of \emph{moving poles}.

Then we give, as an application, another proof of Theorem \ref{PCcase} using this convergence property.

We recall that given a proper sub-domain $D$ of the one point compactification $\C_\infty$ of $\C$ the \emph{Green function} of $D$ is the unique function $G_D:D\times D\rightarrow ]-\infty,\infty]$ such that $G_D(\cdot,\z)$ is harmonic in $D\setminus \{\z\}$ and bounded out from any neighbourhood of $\z$, $G_D(\cdot,\z)$ has a logarithmic pole at $\z$ and $\lim_{z\to z_0}G_D(z,w)=0$ for all $z_0\in \partial D\setminus N$ where $N$ is a \emph{polar} set.  

Let $K\subset\C$ be any compact set, then we can consider the standard splitting in connected components
$$\C\setminus K:=\;\Omega_K\;\bigcup\;\left(\cup_{j\in I}\Omega_j\right),$$
where $\Omega_j$'s are open bounded, while $\Omega_K$ is the only unbounded connected component of $\C\setminus K$. 

To simplify the notation from now on we denote by $g_K(z,\z)$ the Green function $G_{\Omega_K}(z,\z)$ of the unbounded \emph{domain} $\Omega_K,$ notice that the notation is consistent with the case when $\z=\infty.$

There exists another characterization of the Green function that allows also a generalization to several complex variables. Namely one considers the Lelong class $\mathscr L(\C)$ of all subharmonic functions on the complex plane having a logarithmic pole at $\infty,$ e.g.,  $u(z)-\log^+|z|$ is bounded on any neighbourhood of $\{\infty\}$. Then the \emph{extremal subharmonic function} is introduced
\begin{equation}\label{VK}
V_K(z):=\sup\{u\in \mathscr L(\C), u|_{K}\leq 0\}.
\end{equation}
The upper envelope defining $V_K$ has been proved to be equal to the logarithm of the Siciak-Zaharyuta function
$$\Phi_K:=\sup\{|p(z)|^{1/\degree(p)}, p\in \wp(\C),\|p\|_K\leq 1\}.$$ 
By these definitions it follows the Bernstein Walsh Inequality
\begin{equation}
\label{BWI}|p(z)|\leq \|p\|_K\exp(\degree(p) V_K(z))\;,\;p\in \wp(\C).
\end{equation}

Moreover, it turns out that the upper semi-continuous regularization 
$$V^*_K(z):=\limsup_{\z\to z}V_K(\z)$$
coincides with $g_K(z,\infty)$ for all non polar compact $K$; see \cite[Sec. 3]{levnotes}.

Lastly, we recall that a compact set $K$ is said to be \emph{regular} if $g_K(\cdot,\infty)$ (or equivalently $V_K^*$) is continuous on $K$ and hence on $\C.$

We will make repeated use of this classical result (see for instance \cite{Rans}) 
\begin{equation}\label{greenextremal}
g_{K}(z,a)=g_{\eta_a(K)}(\eta_a(z),\infty),
\end{equation}
where $ \eta_a(z)=\frac{1}{z-a}$.

From now on we use the following notation, given any compact set $K$ and a positive $\epsilon$ we set
$$K^\epsilon:=\{z:d(z,K)\leq \epsilon\},$$
where $d(z,K):=\min_{\z\in K}|z-\z|$ is the standard euclidean distance.
\begin{theorem}\label{approximationtheorem}
 Let $K\subset\C$ be a regular compact set and $P$ a compact subset of $\Omega_K$.
Then there exists an open bounded set $D$ such that $K\subset D$ and $P\cap \overline 
D=\emptyset,$ such that for any sequence $\{K_j\}$ of compact subsets of $K$ the 
following are equivalent.
\begin{align}
 &\lim_j \capa(K_j)=&\capa(K).& &\tag{\emph{i}}\label{capaconvergence}&&\\
 &\lim_j g_{K_j}(z,a)=&\,g_K(z,a)&\;\;\text{ loc. unif. for } z\in 
D\;,\;\text{unif. for } a\in P&\tag{\emph{ii}}&&\label{greenconvergence}.
\end{align}
\end{theorem}
In order to prove Theorem \ref{approximationtheorem} we need the following proposition.
\begin{proposition}\label{addedprop}
Let $K\subset \C$ be a regular compact set and $\{K_j\}$ a sequence of compact subsets of $K$, let $D$ be a smooth bounded domain such that $K\subset D$ and $f:D\to \C$ a bi-holomorphism on its image.  Suppose that $\lim_j \capa(K_j)=\capa(K).$ Then
\begin{enumerate}[i)]
\item $g_{K_j}(z,\infty)\to g_{K}(z,\infty)$ locally uniformly,
\item $g_{f(K_j)}(z,\infty)\to g_{f(K)}(z,\infty)$ locally uniformly and
\item $\lim_j\capa (f(K_j))=\capa (f(K)).$
\end{enumerate} 
\end{proposition}
\begin{proof}
It follows by the hypothesis on convergence of capacities that $\mu_{K_j}\rightharpoonup^*\mu_K$, see for instance \cite[Proof of Th. 4.2.3]{StaTo92}.

Let us pick any sequence $\{z_j\}$ of complex numbers converging to $\hat z\in \C$, it follows by the Principle of Descent \cite[Th. 6.8]{SaTo97} that 
$$\limsup_j -U^{\mu_{K_j}}(z_j)\leq -U^{\mu_K}(\hat z).$$

On the other hand, due to regularity of $K$, the fact that $K_j\subset K$ for all $j$ and since by assumption the sequence $-\log \capa (K_j)$ does have limit, we have
\begin{align}
&g_{K}(\hat z,\infty)= \liminf_j g_{K}(z_j,\infty)\leq \liminf_j g_{K_j}(z_j,\infty)\leq \limsup_j g_{K_j}(z_j,\infty)\nonumber\\
=&\limsup_j -U^{\mu_{K_j}}(z_j)-\log \capa(K_j)=\limsup_j -U^{\mu_{K_j}}(z_j)-\log \capa(K)\nonumber\\
\leq& -U^{\mu_K}(\hat z)-\log\capa(K)=g_{K}(\hat z,\infty).\label{Equationalmostuniform}
\end{align}
Thus equality holds, moreover, since the sequence and the limit point are arbitrary we get $g_{K_j}(\cdot,\infty)\to g_{K}(\cdot,\infty)$ locally uniformly in $\C$. Indeed, we can pick any compact set $L\subset \C$ and any maximizing\footnote{Notice that $g_{K_j}(z,\infty)\leq g_K(z,\infty)$ at any $z\in \C$ and, by the continuity of $g_K(\cdot,\infty),$ the function $|g_{K_j}(z_j,\infty)-g_K(z_j,\infty)|=g_{K_j}(z_j,\infty)-g_K(z_j,\infty)$ is upper semi continuous, thus it achieves its maximum on $L$.} sequence $\{z_j\}$ of points in $L$ for $|g_{K_j}(z,\infty)-g_K(z,\infty)|$, i.e., $g_{K_j}(z_j,\infty)-g_K(z_j,\infty)=\max_{z\in L} g_{K_j}(z,\infty)-g_K(z,\infty)$, and notice that extracting a converging subsequence of $z_{j_k}\to\hat z\in L$ and relabelling indexes we have  
\begin{align*}
&\limsup_j \|g_{K_j}(z,\infty)-g_K(z,\infty)\|_L=\limsup_j |g_{K_j}(z_j,\infty)-g_K(z_j,\infty)|\\
\leq& \limsup_j |g_{K_j}(z_j,\infty)-g_K(\hat z,\infty)|+\limsup_j |g_{K}(z_j,\infty)-g_K(\hat z,\infty)|\\
=&\limsup_j |g_{K_j}(z_j,\infty)-g_K(\hat z,\infty)|=0.
\end{align*}
Here we used both the continuity of $g_K(\cdot,\infty)$ and \eqref{Equationalmostuniform}.

Now we introduce some tools that are classical in (pluri-)potential theory in several complex variables. The one variable counterparts of these notions are just normalizations by a negative scaling factor: this leads to consider $\sup$ in place of $\inf$ and superharmonic functions in place of subharmonic. We choose this setting because it is easier to provide a proof of the above statement in this notation; we refer the reader to \cite[Ch. II.5]{SaTo97} for the one variable definitions and properties.

We pick a domain $D$ containing $K$ and we define the relative extremal subharmonic function
\begin{equation}\label{EquationDefinitionRelativeExtremal}
U_{K,D}^*(z):=\limsup_{\z\to z}\sup\{u(\z)\in shm(D),u\leq 0,\, u|_K\leq -1\}.
\end{equation}
Here $shm(D)$ stands for the set of subharmonic functions on $D$.
This is a subharmonic function on $D$ whose distributional Laplacian is a positive measure supported on $K$, moreover $U_{K,D}^*\equiv -1$ q.e. on $K$ for an arbitrary compact set $K$ and $U_{K,D}^*\equiv -1$ for a regular compact set $K$; see \cite{BeTa82}. The reader is invited to compare this to the Green potential of the condenser $(K,\partial D)$ in \cite[Ch. II.5]{SaTo97}.

The function $U_{K,D}^*-1$ solves the following variational problem that defines the \emph{relative capacity} of $K$ in $D.$ 
\begin{equation}\label{EquationRelativecapacity}
\capa(K,D):=\sup\left\{\int_K \Delta u: u\in shm(D,[0,1])\right\},
\end{equation}
namely one has $\capa(K,D)=\int_K \Delta U_{K,D}^*=\int_K -U_{K,D}^* \Delta U_{K,D}^*.$ 

Now we show that $U_{K_j,D}^*\to U_{K,D}^*$ uniformly on $D.$

On one hand, by the definition \eqref{EquationDefinitionRelativeExtremal} above, we have
$$\frac{g_{K_j}(z,\infty)}{\|g_{K_j}(\cdot,\infty)\|_D}-1\leq  U_{K_j,D}^*(z)\;\forall z\in D.$$

On the other hand, by the estimate $g_{K_j}(z)\geq (\inf_{\partial D}g_{K_j})(U_{K_j,D}^*(z)+1)$ for all $z\in F$ (see \cite[Prop. 5.3.3]{Kli}), it follows that
$$  U_{K_j,D}^*(z)\leq \frac{g_{K_j}(z,\infty)}{\inf_{w\in \partial D}g_{K_j}(w,\infty)}-1\;,\;\;\forall z\in D.$$
Hence we have
$$\frac{g_{K_j}(z,\infty)}{\|g_{K_j}(\cdot,\infty)\|_D}-1\leq  U_{K_j,D}^*(z)\leq \frac{g_{K_j}(z,\infty)}{\inf_{w\in \partial D}g_{K_j}(w,\infty)}-1\;\forall z\in D.$$
Since we proved that $g_{K_j}(z,\infty)\to g_{K}(z,\infty)$ uniformly, we get $U_{K_j,D}^*\to U_{K,D}^*$ uniformly on $D.$

It follows by the above convergence that $\capa(K_j,D)\to \capa(K,D)$ as well; see the definition \eqref{EquationRelativecapacity} of relative capacity and lines below. To show that we simply pick $\phi\in \mathscr C^\infty_c(D,[0,1])$ such that $\phi\equiv 1$ in a neighbourhood of $K$ and we write
\begin{align*}
\capa(K,D)&=\int_{K}\Delta U_{K,D}^*= \int_{D}\phi\Delta U_{K,D}^*=\int_{D}\phi\Delta U_{K,D}^*\\
&=\lim_j \int_{D}\phi\Delta U_{K_j,D}^*=\lim_j\capa(K_j,D).
\end{align*}
Now we note that, given a biholomorphism $f$ of $D$ on the smooth domain $f(D)=\Omega\subset \C$ there is a one to one correspondence between functions in $\{u \in shm(D):u\leq 0,u|_G\leq -1\}$ and $\{v \in shm(\Omega):v\leq 0,v|_{f(G)}\leq -1\}$ for any compact set $G\subset D.$ For this reason, setting $F=f(K)$ and $F_j=f(K_j)$ , one has $U_{F_j,\Omega}^*\equiv U_{K_j,D}\circ f$ and $U_{F,\Omega}^*\equiv U_{K,D}\circ f.$ Therefore we have
$$U_{F_j,\Omega}^*\to U_{F,\Omega}^*\;\;\text{locally uniformly in }\Omega.
 \text{ and}$$
 $$\capa(F_j,\Omega)\to \capa(F,\Omega).$$

Let us recall that we can find a constant $A>0$ such that $\sup_{\Omega} g_{F_j}(z,\infty)\leq \frac{A}{\capa(F_j,\Omega)}$ for each subset of the compact set $F$; see \cite{AlTa84}. Thus we can pick $j_0$ such that, for $j\geq j_0$, we have $\sup_{\Omega} g_{F_j}(z,\infty)\leq \frac{2A}{\capa(F,\Omega)}=M.$

It follows by the definition of relative extremal function that we have
$$0\leq \frac{g_{F_j}(z,\infty)}{M}-1\leq U_{F_j,\Omega}^*(z),\forall j>j_0, \forall z\in \Omega.$$
But since the right hand side converges uniformly to $-1$ on $F$ we get that $g_{F_j}(z,\infty)\to 0$ uniformly on $F.$ Note that the same reasoning shows that in particular $g_F(z,\infty)\equiv 0$ on $F$, that is $F$ is regular.

In particular for any $\epsilon>0$ we can pick $j_\epsilon$ such that for any $j>j_\epsilon$ we have $g_{F_j}(z,\infty)-\epsilon \leq 0\equiv g_{F}(z,\infty)$ for any $z\in F.$

Now we recall that for any compact set $L\subset \C$ the Green function with logarithmic pole at $\infty$ can be expressed as the upper semi continuous regularization of the upper envelopes of all subharmonic functions on $\C$ in the Lelong class (e.g., locally bounded and having a logarithmic pole at infinity) that are bounded above by zero on $L$; see \eqref{VK} above, \cite[Sec. 3]{levnotes} and \cite{Rans}. Hence, we get  $g_{F_j}(z,\infty)-\epsilon \leq g_{F}(z,\infty), \forall j>j_\epsilon, z\in \C.$
Similarly, $g_{F}(z,\infty) \leq g_{F_j}(z,\infty), \forall j\in \N, z\in \C,$ since $F_j\subset F.$

Therefore we have
$$g_{F_j}(z,\infty)-\epsilon \leq g_{F}(z,\infty)  \leq g_{F_j}(z,\infty), \forall j>j_\epsilon, z\in \C.$$
That is $g_{F_j}(z,\infty)\to g_{F}(z,\infty)$ uniformly in $\C.$

It follows by this uniform convergence that $\mu_{F_j}\rightharpoonup^*\mu_{F}$ (note that $\mu_{F}=\Delta g_{F}(z,\infty)$ and the distributional Laplacian, by linearity,  is continuous under the local uniform convergence) and thus $U^{\mu_{F}}= \lim_j U^{\mu_{F_j}}$ uniformly on compact sets of $\C\setminus F$ (by the uniform continuity of the $\log$ kernel away from $0$), thus in particular $U^{\mu_{F}}(\hat z)= \lim_j U^{\mu_{F_j}}(\hat z)$ for any given $\hat z\in \C\setminus F.$

Now we have, for any $\hat z\in \C\setminus F$
\begin{align*}
&-\log\capa(F_j)\\
=&g_{F_j}(\hat z,\infty)+U^{\mu_{F_j}}(\hat z)\to g_{F}(\hat z,\infty) +U^{\mu_{F}}(\hat z)\\
=&-\log\capa(F).
\end{align*}
\qed
\end{proof}

\begin{proof}[of Theorem \ref{approximationtheorem}]
 By Hilbert Lemniscate Theorem for any $\epsilon<d(K,P):=\inf_{z\in K}d(z,P)$ we can pick a 
polynomial $q$ such that 
 $$\hat K\subset D:= \{|q|<\|q\|_K\}\subset \hat K^\epsilon\;,\,\hat K^\epsilon\cap P=\emptyset.$$ 
Let $D$ be fixed in such a way.

We introduce a more concise notation for the Green functions involved in the proof:
we denote by $g(z,a)$ the Green function with pole at $a$ for the set $\Omega_K$, we omit the pole when $a=\infty$, we add a subscript $j$ if $K$ is replaced by $K_j$ and a superscript $b$ if $K$ or $K_j$ are replaced by $\eta_b(K)$ or $\eta_b(K_j)$, where $\eta_b(z):=1/(z-b).$ In symbols
\begin{equation*}
\begin{array}{ccc}
g(z):=g_K(z,\infty)&,&g_j(z,a):=g_{K_j}(z,a),\\
g_j(z):=g_{K_j}(z,\infty)&,&g^b(z,a):=g_{\eta_b K}(z,a),\\
g(z,a):=g_{K}(z,a)&,&g_j^b(z,a):=g_{\eta_b K_j}(z,a).\\
\end{array}
\end{equation*}
Moreover we set $E_j:=\eta_{a_j}(K_j)$ and $E:=\eta_{\hat a}(K).$
\emph{Proof of \textbf{\eqref{capaconvergence}}} $\Rightarrow$ 
\emph{\textbf{\eqref{greenconvergence}}}.
In order to prove the local uniform convergence of $g_j(\cdot,a)$ to $g(\cdot,a)$, uniformly with respect to $a\in P$, we pick any converging sequence $P\ni a_j\to \hat a$, we set $\tilde D:=\eta_{\hat a}(D)$ and we prove
\begin{equation}\label{claim1}
g^{a_j}_j\to  g^a\;\text{ loc. unif. in } \tilde 
D.
\end{equation}
Finally we notice that $g_j(\cdot,a_j)=g^{a_j}_j\circ \eta_{a_j}^{-1}\to  g^a\circ \eta_{\hat a}^{-1}=g(\cdot,\hat a)\;\text{ loc. unif. in } D$ hence the result follows.

We proceed along the following steps:
\begin{align}
& \lim_j\capa(E_j)=\capa(E). \tag{S1}\label{capacity}\\
& \mu_{E_j}\rightharpoonup^*\mu_{E}. \tag{S2}\label{weakconvergence}\\
& \lim_jg_{E_j}(z,\infty)=g_E(z,\infty),\;\text{ loc. unif. in }\C.\tag{S3}\label{uniformgreen}
\end{align}

Here we used the standard notation (see \eqref{equilibriummeasure}) $\mu_E$ for the equilibrium measure of the compact non-polar set $E$.

To prove \eqref{capacity} we use \cite[Th. 5.3.1]{Rans} applied to the set of 
maps $\phi_j:=\eta_{a_j}\circ\eta_{\hat a}^{-1}$ and $\psi_j:=\phi_j^{-1}$ 
together with the assumption \eqref{capaconvergence}. Each map is bi-holomorphic on 
a neighbourhood of $\tilde D,$ moreover we have

\begin{equation}
\begin{split}
 \|\phi_j'\|_{\eta_{\hat a}(K)}&=\max_{\z\in \eta_{\hat a}(K) } 
\frac 1 {\left|1+(\hat a-a_j)\z\right|^2}\\
&\leq \max_K \frac{|z-\hat a|^2}{|z-a_j|^2}\leq 1+\frac{|\hat a-a_j|^2}{|\dist(K,P)|^2}=:L_j.\label{LipMap}
\end{split}
\end{equation}
\begin{equation}
\begin{split}
\|\psi_j'\|_{\eta_{\hat a_j}(K_j)}&=\left(\min_{\z\in \eta_{a_j}(K_j) } 
\left|1+(a_j-\hat a)\z\right| \right)^{-2}\\
&\leq \max_{K_j} \frac{|z-a_j|^2}{|z-\hat a|^2}\leq 1+\frac{|\hat a-a_j|^2}{|\dist(K,P)|^2}=L_j.\label{LipInverse}
\end{split}
\end{equation}
We denoted by $\dist(K,H):=\inf\{\epsilon>0: K^\epsilon\supseteq H\,,\,H^\epsilon\supseteq K\}$ the Hausdorff distance of $K$ and $H$. Notice that $L_j\to 1$ as $j\to \infty.$

We recall that $\capa(f(E))\leq \Lip_E(f)\capa(E),$ where $\Lip_E(f):=\inf\{L: |f(x)-f(y)|<L|x-y|\,\forall x,y\in E\}$ for any Lipschitz mapping $f:E\rightarrow \C;$ \cite{Rans}[Th. 5.3.1]. Therefore, due to \eqref{LipMap} and \eqref{LipInverse}, we have the following upper bounds.
\begin{align*}
 \capa(E_j)=&\capa(\phi_j( \eta_{\hat a }(K_j)))\leq 
L_j\capa(\eta_{\hat a }(K_j)),\\
\capa(\eta_{\hat a}(K_j))=&\capa(\eta_{\hat 
a}\circ\eta_{a_j}^{-1}(E_j))=\capa(\psi_j(E_j))\leq 
L_j\capa(E_j).
\end{align*}
Thus, using $\lim_j L_j=1$, we have
\begin{align}
\liminf_j \capa(E_j)\geq & \liminf_j \frac 1 {L_j}\capa(\eta_{\hat a }(K_j))= \liminf_j \capa(\eta_{\hat a }(K_j)),\label{liminf}\\
\limsup_j \capa(E_j)\leq & \limsup_jL_j \capa(\eta_{\hat a }(K_j))\leq \limsup_j \capa(\eta_{\hat a }(K_j)).\label{limsup}
\end{align}
Now we use Proposition \ref{addedprop} to get $\lim_j \capa(\eta_{\hat a }(K_j))=\capa(\eta_{\hat a}(K))$ and thus
$$\liminf_j \capa(E_j)\geq \capa(\eta_{\hat a }(K))\geq \limsup_j \capa(E_j),$$
hence all inequalities are equalities and $\lim_j \capa(E_j)=\capa(E);$ this concludes the proof of \eqref{capacity}.

The proof of \eqref{weakconvergence} is by the Direct Method of Calculus of Variation. More explicitly, let $\mu_j:=\mu_{E_j}$ be the sequence of equilibrium measures, i.e., the  minimizers of $I[\cdot]$ as defined in \eqref{logenergy} among the classes $\mu\in \mathcal M_1(E_j)$. From \eqref{capacity} it follows that $\liminf_jI[\mu_j]=I[\mu_E]$. Therefore, if $\mu$ is any weak$^*$ closure point of the sequence, by lower semi-continuity of $I$, we get $I[\mu]\leq I[\mu_E].$

Notice that without loss of generality we can assume $K_j$, and thus $E_j$, to be not polar, since $\capa(K_j)>0$ for $j$ large enough.

If $\support \mu\subseteq E$, by the strict convexity of the energy functional, we have that $\mu=\mu_{E}$ and the whole sequence is converging to $\mu_E$; see \cite[Part I, Th. 1.3]{SaTo97}. Then we are left to prove $\support \mu\subseteq E$, this follows by the uniform convergence of $\eta_{a_j}$ to $\eta_{\hat a}$ and by properties of weak$^*$ convergence of measures.

To this aim, we suppose by contradiction $ \support \mu\cap \left( \C\setminus E\right)\neq \emptyset$. It follows that there exists a Borel set $B\subset \C\setminus E$ with $\mu(B)>0.$ Since $\mu$ is Borel we can find a closed set $C\subset B$ still having positive measure. Being $\C$ a metric space and we can find an open neighbourhood $A$ of $C$ disjoint by $E$ with $\mu(A)>0$.

Due to the Portemanteau Theorem (see for instance \cite[Th. 2.1]{Bi99}) we have
$$0<\mu(A)\leq \liminf_j \mu_j(A).$$
Therefore $C\subseteq A\subset E_{j_m}$ for an increasing subsequence $j_m.$

By the uniform convergence $\eta_{a_{j_m}}\to \eta_{\hat a}$ it follows that $C\subseteq A\subseteq E$, a contradiction since we assumed $C\cap E=\emptyset.$
 
Let us prove \eqref{uniformgreen}.

First, we recall (see for instance \cite[pg. 53]{SaTo97}) that for any compact set $M\subset \C$ we have $g_M(z,\infty)=-\log \capa(M)-U^{\mu_M}(z)$. Hence it follows that
\begin{equation}
g_{E_j}(\z,\infty)=-\log\capa(E_j)-U^{\mu_j}(\z).\label{decompose}
\end{equation} 

Due to \eqref{weakconvergence} and by the Principle of Descent \cite[I.6, Th. 6.8]{SaTo97} for any $\z\in \C$ we have 
\begin{equation}
 \limsup_j -U^{\mu_j}(\z)\leq -U^{\mu_E}(\z).\label{descent}
\end{equation}
It follows by \eqref{capacity},\eqref{decompose} and \eqref{descent} that
\begin{equation*}
\limsup_j g_{E_j}(\z,\infty)\leq g_E(\z,\infty),\;\forall \z\in \C.
\end{equation*}
The sequence of subharmonic functions $\{ g_{E_j}(\z,\infty)\}$ is locally uniformly bounded above and non negative, therefore we can apply the Hartog's Lemma. For each $\epsilon>0$ there exists $j(\epsilon)\in \N$ such that
\begin{equation*}
\|g_{E_j}(\z,\infty)\|_E \leq \|g_E(\z,\infty)\|_E+\epsilon=\epsilon. 
\end{equation*} 
Here the last equality is due to the regularity of $K$ and thus of $E$ (e.g. $g_E(\z,\infty)\equiv 0$ $\forall\z \in E$).
Therefore we have  
\begin{equation}
g_{E_j}(\z,\infty)-\epsilon \leq g_E(\z,\infty)\;,\;\forall \z\in E.\label{upperboundonE} 
\end{equation} 

By the extremal property of the Green function (see \eqref{VK} and lines below) and the upper bound \eqref{upperboundonE} it follows that
\begin{equation}
g_{E_j}(\z,\infty)-\epsilon \leq g_E(\z,\infty)\;,\;\forall \z\in \C,j\geq j(\epsilon).\label{upperboundonC} 
\end{equation}

Since $g_E(\cdot,\infty)$ is continuous (hence uniformly continuous on a compact neighbourhood $M$ of $E$ containing all $E_j$) for any $\epsilon>0$ we can pick $\delta>0$ such that $g_E(\z,\infty)\leq \epsilon$ for any $\z\in E^\delta$.

Let us set $j'(\epsilon):=\min\{\bar j : E_j\subseteq E^\delta \forall j\geq \bar j\}$, notice that $j'(\epsilon)\in \N$ for any (sufficiently small) $\epsilon>0$ since 
$$E_j\subset\eta_{a_j}(K)\subseteq L_j \eta_{\hat a} (K)=L_j E\subseteq E^{(L_j-1)\|z\|_E},$$
where $L_j$ is defined in equations \eqref{LipMap} \eqref{LipInverse} and $L_j\to 1.$

It follows by this choice that 
$$\|g_{E}(\z,\infty)\|_{E_j}\leq \epsilon,\;\forall j\geq  j'(\epsilon).$$ Therefore, again by the extremal property of $g_{E_j}(\z,\infty)$, we have
\begin{equation}
 g_{E}(\z,\infty)-\epsilon\leq g_{E_j}(\z,\infty),\;\;\forall \z\in \C, j\geq j'(\epsilon).\label{lowerboundonC}
\end{equation}

Now simply observe that \eqref{lowerboundonC} and \eqref{upperboundonC} imply 
$$ g_{E}(\z,\infty)-\epsilon\leq g_{E_j}(\z,\infty)\leq g_{E}(\z,\infty)+\epsilon\;,\;\;\forall j\geq \max\{j(\epsilon),j'(\epsilon)\}.$$
Therefore $g_{E_j}(\cdot,\infty)$ converges locally uniformly to $g_{E}(\cdot,\infty).$

To conclude the proof of \eqref{capaconvergence} $\Rightarrow$ \eqref{greenconvergence} let us pick any compact subset $L$ of $D$.
\begin{align*}
&\|g_j^{a_j}-g^{\hat a}\|_L=\|g_{E_j}(\eta_{a_j}(z),\infty)-g_E(\eta_{\hat a}(z),\infty)  \|_L\leq\\
&\|g_{E_j}(\eta_{a_j}(z),\infty)-g_E(\eta_{a_j}(z),\infty)  \|_L +\|g_{E}(\eta_{a_j}(z),\infty)-g_E(\eta_{\hat a}(z),\infty)  \|_L\to 0
\end{align*}
Here we used the continuity of $g_E(z,\infty)$ and the local uniform convergence of $\eta_{a_j}$ to $\eta_{\hat a}.$ By the arbitrariness of the sequence of poles $\{a_j\}$ \eqref{greenconvergence} follows.
\qed
\emph{Proof of \textbf{\eqref{greenconvergence}}} $\Rightarrow$ 
\emph{\textbf{\eqref{capaconvergence}}.}
Fix any pole $a\in P$ and set $\eta_a(z):=\frac 1{z-a}$, $E:=\eta_a(K),$ $E_j:=\eta_a(K_j)$, by our assumption we have $g_j^a\to g^a$ locally uniformly in $\C$ thus
$$g_{E_j}(\cdot,\infty)\to g_{E}(\cdot,\infty),$$ 
uniformly on some neighbourhood $D$ of $E$  (where $\eta_a^{-1}$ is a biholomorphism on its image).

It follows that $\mu_{E_j}\rightharpoonup^*\mu_E.$ Let us pick a point $\hat z\in D\setminus E$, by uniform continuity of the $\log$ kernel away from $0$ we have $U^{\mu_{E_j}}(\hat z)\to U^{\mu_{E}}(\hat z)$. On the other hand $g_{E_j}(\hat z,\infty)\to g_{E}(\hat z,\infty),$ therefore we have
\begin{align*}
&\lim_j -\log\capa(E_j)=\lim_j \left(g_{E_j}(\hat z,\infty)+ U^{\mu_{E_j}}(\hat z)\right)\\
=& g_{E}(\hat z,\infty)+\lim_j U^{\mu_{E_j}}(\hat z)=-\log\capa(E),
\end{align*}
where existence of the limit is part of the statement  and follows by the existence  of the limits of the two terms of the sum.

We apply Proposition \ref{addedprop} with $f:=\eta_a^{-1}$ to get $-\log\capa(K_j)\to -\log\capa(K).$\qed
\end{proof}
\subsection{A direct proof of the mass density sufficient condition by convergence of Green functions}\label{anotherproof}
Theorem \ref{approximationtheorem} can be used to prove directly Theorem \ref{PCcase}, that is, for measures having regular compact support, the classical sufficient mass density condition in \cite{BlLe13} or $\Lambda^*$ condition \cite{StaTo92} implies a rational Bernstein Markov Property, provided $P\subset \Omega_K.$

\begin{proof}[Direct proof of Theorem \ref{PCcase}]
The proof follows the idea of \cite[Th. 4.2.3]{StaTo92}, except for the lack of the Bernstein Walsh Inequality \eqref{BWI} which is not available for rational functions.

In place of it we use the following variant due to Blatt \cite[eqn. 2.2]{Bl09} which holds for any rational function $r_k$ of the form $r_k(\z) =\frac{p_k(\z)}{q_k(\z)}=\frac{c_k\prod\limits_{j=0}^{m_k}(\z-z_j^{(k)})}{\prod\limits_{j=0}^{n_k}(\z-a_j^{(k)})}$. For $ \z\notin\{a_1,\dots,a_{n_k}\}$ we have 
\begin{equation}
 |r_k(\z)|\leq\|r_k\|_K\exp \left(\sum_{j=1}^{n_k} g_{K}(\z,a_j)+ (m_k-n_k)g_{K}(\z,\infty) 
\right).\label{blatt}
\end{equation}
 Thus in particular we have
 \begin{equation*}
 |r_k(\z)|\leq \|r_k\|_{K_j}\exp\left(n_k \max_{a\in P}g_{K_j}(\z,a)+(m_k-n_k)g_{K_j}(\z,\infty)\right)\;\forall \z\in \C\setminus P.
 \end{equation*}
Notice that, for any sequence $K_j\subset K$ such that $\capa K_j\to \capa K$, from Theorem \ref{approximationtheorem} it follows that 
$$\max_{a\in P}g_{K_j}(\z,a)\to \max_{a\in P}g_{K}(\z,a)\;\text{ locally uniformly in }\C\setminus P.$$
Moreover, it is well known that under the same condition we have
$$g_{K_j}(\z,\infty)\to g_{K}(\z,\infty)\;\text{ locally uniformly in }\C.$$

Pick any $\{r_k\}\in \mathcal R(P).$ By the regularity of $K$ and the compactness of $P$ for any $\epsilon>0$ there exists $\delta>0$ such that
\begin{align*}
 g_{K}(\z,a)&\leq \epsilon\;\;\forall \z: \dist(\z,K)\leq \delta,\;\forall a\in P\\
g_{K}(\z,\infty)&\leq \epsilon\;\;\;\;\;\forall \z: \dist( \z,K)\leq \delta .
\end{align*}
Let us pick $\epsilon>0$, it follows by \eqref{blatt} that there exists $\delta>0$ such that $\forall \z: \dist(\z,K)\leq \delta$ we have
\begin{equation}
|r_k(\z)|\leq \|r_k\|_K e^{\left(n_k \max_{a\in P}g_{K}(\z,a)+(m_k-n_k)g_{K}(\z,\infty) \right)}\leq e^{(k\epsilon)}\|r_k\|_K.\label{byregularity}
\end{equation}

By Theorem \ref{approximationtheorem} (possibly shrinking $\delta$) we have, for any $A\subset K,\text{ with }\capa(A)>\capa(K)-\delta$ and locally uniformly in $\C\setminus P,$ 
\begin{align}
  \max_{w\in P}g_{A}(\z,w) &\leq  \max_{w\in P}g_{K}(\z,w)+\epsilon\;,\\
g_{A}(\z,\infty)&\leq g_{K}(\z,\infty)+\epsilon\;.
  \label{bytheorem}
\end{align}
Using \eqref{byregularity} and \eqref{bytheorem} we have
\begin{equation}
 |r_k(\z)|\leq e^{(2\epsilon k)}\|r_k\|_{A}\;\;\forall \z\in K^\delta,\,\forall A\subset K \text{ with }\capa(A)>\capa(K)-\delta.
\end{equation}
Let $\z_0\in A$ be such that $\|r_k\|_A=|r_k(\z_0)|$, we show that a lower bound for 
$|r_k|$ holds in a ball centred at $\z_0.$ By the Cauchy Inequality we have 
$|r_k'(\z)|<\frac{\|r_k\|_{\overline{B(\z_0,s)}}}{s}\leq\frac{e^{(2\epsilon 
k)}\|r_k\|_{A}}{s}$, for any $|\z-\z_0|<s$, $s<\delta.$ Taking $s=\delta/2$ we can integrate such an estimates as follows $\forall z\in \overline{B(\z_0,\delta/2)}$
\begin{equation*}
 \|r_k\|_{A}=|r_k(\z_0)|=\left|r_k(z)+\int_{[z,\z_0]}r_k'(\z)d\z  \right|\leq 
|r_k(z)|+|z-\z_0|\frac{e^{(2\epsilon k)}\|r_k\|_{A}}{\delta/2}.
\end{equation*}
It follows by the above estimate that  
\begin{equation}
\min_{z\in \overline{B(\z_0,\frac{\delta e^{(-2\epsilon k)}} 4)}}|r_k(z)|\geq 
\frac{\|r_k\|_A}{2}\;\;\forall A\subset K \text{ with } \capa(A)>\capa(K)-\delta.\label{minestimate} 
\end{equation}
Now we provide a lower bound for $L^2_\mu$ norms of $r_k$ by integrating the last inequality on a (possibly smaller ball) and picking $A\subset K$ according to the mass density condition \eqref{massdensity2}.

Precisely, set $\rho_k:=e^{(-3k\epsilon)},$ by the hypothesis we can pick $t>0$ and $A_k\subset K$ with $\capa(A_k)>\capa(K)-\delta$ such that $\mu(B_k):=\mu(\overline{B(\eta,\rho_k)})\geq \rho_k^t$ $\forall \eta\in A_k$. We pick $k\geq \bar k$ such that $\rho_k<\frac{\delta e^{(-2\epsilon k)}} 4$, thus using \eqref{minestimate} we get
\begin{align*}
\|r_k\|_{L^2_\mu}^2&\geq \int_{B_k}|r_k|^2 d\mu\geq \min_{z\in B_k}|r_k(z)|^2\mu(B_k)\geq \frac{\|r_k\|_{A_k}^2}{4}\rho_k^t\\
&\geq\frac{e^{(-3tk\epsilon)}}{4}\|r_k\|_{A_k}^2\geq\frac{e^{(-(4+3t)k\epsilon)}}{4}\|r_k\|_{K}^2.
\end{align*}
It follows that $\left(\frac{\|r_k\|_K}{\|r_k\|_{L^2_\mu}}\right)^{1/k}\leq 4^{1/k} e^{((4+3t)\epsilon)},$ by arbitrariness of $\epsilon>0$ we can conclude that
$$\limsup_k \left(\frac{\|r_k\|_K}{\|r_k\|_{L^2_\mu}}\right)^{1/k}\leq 1$$\qed
\end{proof}

\section{Application: a $L^2$ meromorphic Bernstein Walsh Lemma}\label{appl}
For a given compact set $K\subset \C$ we denote by $D_r$ the set $\{z\in \C: g_K(z,\infty)<\log r\}$ and by $\mathscr M_{n}(D_r)$ the class of meromorphic functions having precisely $n$ poles (counted with  their multiplicities) in $D_r.$ Let us denote by $\mathcal R_{k,n}$ the class of rational functions having at most $k$ zeroes and at most $n$ poles (each of them counted with its multiplicity).

It follows by the work of Walsh \cite{WAL}, Saff \cite{Sa71} and Gonchar \cite{Go75} that, given a function $f\in \mathscr C(K)$, where $K$ is a compact regular set, $f$ admits a meromorphic extension $\tilde f\in \mathscr M_{n}(D_r)$ if and only if one has the \emph{overconvergence} of the best uniform norm approximation by rational functions with $n$ poles, that is 
\begin{equation}\label{walshcond}
\limsup_k d_{n,k}(f,K)^{1/k}:=\limsup_k \inf_{r\in \mathcal R_{k,n}}\|f-r\|_K^{1/k}\leq 1/r.
\end{equation}
In the case of a finite measure $\mu$ having compact support $K$ and such that $(K,\mu,P)$ has the rational Bernstein Markov property for any compact set $P$, $P\cap K=\emptyset$, one can rewrite such a theorem checking the overconvergence of best $L^2_\mu$ rational approximations instead of best uniform ones. Notice that if $K=\hat K$ any Bernstein Markov measure supported on $K$ has such a property. More precisely, we can prove the following in the spirit of \cite[Prop. 9.4 ]{levnotes}, where we use the notation $\sing(f)$ to denote the set of poles of the function $f$.

\begin{theorem}[$L^2$ Meromorphic Bernstein Walsh Lemma]
Let $K$ be a compact regular subset of $\C$, let $f\in \mathscr C(K)$ and let $r>1$.
The following are equivalent.
\begin{enumerate}[i)]
\item There exists $\tilde f\in \mathscr M_n(D_r)$ such that $\tilde f|_K\equiv f.$
\item $\limsup_k d_{k,n}^{1/k}(f,K)\leq 1/r.$
\item For any finite Borel measure $\mu$ such that $\support \mu=K$ and $(K,\mu,P)$ has the rational Bernstein Markov property for any compact set $P$ such that $P\cap K=\emptyset$, denoting by $r_{k,n}^\mu$ a best $L^2_\mu$ approximation to $f$ in $\mathcal R_{k,n}$, one has 
$$\limsup_k\left(\|f-r_{k,n}^\mu\|_K\right)^{1/k}\leq 1/r,$$
 provided that $\overline{\{\sing(r_{k,n})\}_k}\cap K=\emptyset.$
\item With the same hypothesis and notations as in \emph{iii)} we have 
$$\limsup_k\left(\|f-r_{k,n}^\mu\|_{L^2_\mu}\right)^{1/k}\leq 1/r,$$
provided that $\overline{\{\sing(r_{k,n})\}_k}\cap K=\emptyset.$
\end{enumerate}
\end{theorem} 
\begin{proof}(\emph{i} $\Leftrightarrow$ \emph{ii.}) The theorem has been proven in \cite{Go75}, see also \cite{BeCa13}.

(\emph{ii} $\Rightarrow$ \emph{iii}.) Let us pick $\rho>r$, we find $C>0$ such that
$$d_{k,n}^{1/k}(f,K)\leq C/\rho^k,\;\;\forall k.$$
Let us pick $s_{k,n}\in \mathcal R_{k,n}$ such that $\|f-s_{k,n}\|_K=d_{k,n}(f,K)$ and set $P_\infty=\overline{\{\sing(s_{k,n})\}_k}$. Notice that
\begin{align}\label{L2estimate}
&\|f-r_{k,n}^\mu\|_{L^2_\mu}\leq \|f-s_{k,n}\|_{L^2_\mu}\leq \mu(K)^{-1/2}\|f-s_{k,n}\|_K\\
=&\mu(K)^{-1/2} d_{k,n}(f,K)\leq \mu(K)^{-1/2}C/\rho^k.\nonumber
\end{align}
In particular it follows that
\begin{equation*}
\|r_{k,n}^\mu-r_{k-1,n}^\mu\|_{L^2_\mu}\leq \|f-r_{k,n}^\mu\|_{L^2_\mu}+\|f-r_{k-1,n}^\mu\|_{L^2_\mu}\leq \frac{\mu(K)^{-1/2}C(1+\rho)}{\rho^k}.
\end{equation*}
We apply the rational Bernstein Markov property to $(K,\mu,P )$, with $P:=P_\infty\cup P_2$, $P_2=\overline{\{\sing(r_{k,n})\}_k}$, in the following equivalent formulation, for any $\epsilon>0$ there exists $M=M(\epsilon,K,\mu,P)$ such that $\|s\|_K\leq M(1+\epsilon)^k\|s\|_{L^2_\mu}$ for any $s\in \mathcal R_{k,n}$, $\sing s\subset P$, $n\leq k$, $\forall k$. Notice that $P_\infty\cap K=\emptyset$ follows by the assumption $\limsup_k d_{k,n}^{1/k}(f,K)\leq 1/r$; \cite{WAL}. We get
\begin{equation}\label{Kestimate}
\|r_{k,n}^\mu-r_{k-1,n}^\mu\|_K\leq M\mu(K)^{-1/2}C(1+\rho) \left(\frac{1+\epsilon}{\rho}\right)^k.
\end{equation}
By equation \eqref{L2estimate} $r_{k,n}^\mu\to f$ in $L^2_\mu$, therefore some subsequence converges almost everywhere with respect to $\mu$. By equation \eqref{Kestimate} we can show that the sequence of functions $\{r_{k,n}\}$ is a Cauchy sequence in $\mathscr C(K)$ thus it has a uniform continuous limit $g$. Therefore $f\equiv g$ and the whole sequence is uniformly converging to $f$ on $K$. Notice that $f\equiv g$ on a carrier of $\mu$, thus on a dense subset of the support $K$ of $\mu$.

Now notice that
\begin{align*}
&\|f-r_{k,n}\|_K\leq \left\|\sum_{j=k+1}^\infty r_{j,n}^\mu-r_{j-1,n}^\mu\right\|_K \leq\sum_{j=k+1}^\infty \|r_{j,n}^\mu-r_{j-1,n}^\mu\|_K\\
\leq& M\mu(K)^{-1/2}C(1+\rho) \sum_{j=k+1}^\infty\left(\frac{1+\epsilon}{\rho}\right)^j= M\mu(K)^{-1/2}C(1+\epsilon)\frac{1+\rho}{\rho-1}\left(\frac{1+\epsilon}{\rho}\right)^k.
\end{align*}  
Therefore we have
$$\limsup_k \|f-r_{k,n}\|_K^{1/k}\leq \limsup_k \left(\frac{M\mu(K)^{-1/2}C(1+\epsilon)(1+\rho)}{\rho-1}\right)^{1/k}\frac{1+\epsilon}{\rho}=\frac{1+\epsilon}{\rho}.$$
The thesis follows letting $\epsilon\to 0^+$ and $\rho\to r^+.$

(\emph{iii} $\Rightarrow$ \emph{ii}.) By definition one has
\begin{align*}
1/r&\geq \limsup_k\left(\|f-r_{k,n}^\mu\|_K\right)^{1/k}\geq \limsup_k\left(\inf_{r\in \mathcal R_{k,n}}\|f-r\|_K\right)^{1/k}\\
&=\limsup_k d_{k,n}^{1/k}(f,K).
\end{align*}
(\emph{iii} $\Rightarrow$ \emph{iv}.)
Simply notice that
\begin{align*}
1/r&\geq \limsup_k\left(\|f-r_{k,n}^\mu\|_K\right)^{1/k}\geq \limsup_k\left(\mu(K)^{1/2}\|f-r_{k,n}^\mu\|_{L^2}\right)^{1/k}\\
&=\limsup_k\left(\|f-r_{k,n}^\mu\|_{L^2}\right)^{1/k}.
\end{align*}

(\emph{iv} $\Rightarrow$ \emph{iii}.) This implication can be proven using a similar reasoning to the one of (\emph{ii} $\Rightarrow$ \emph{iii}).

The sequence $r_{k,n}$ is converging to $f$ in $L^2_\mu$ by assumption, then there exists a subsequence converging to $f$ almost everywhere.

Due to the rational Bernstein Markov property of $\mu$ with respect to $K$ and $P_2$ we have
$$\|r_{k,n}-r_{k-1,n}\|_K\leq M(1+\epsilon)^k\|r_{k,n}-r_{k-1,n}\|_{L^2_\mu}$$
and we can estimate the right hand side as follows 
$$\|r_{k,n}-r_{k-1,n}\|_{L^2_\mu}\leq \|r_{k,n}-f\|_{L^2_\mu}+\|f-r_{k-1,n}\|_{L^2_\mu}\leq C/\rho^k(1+\rho)$$
for a suitable $C>0$ and $\rho>r$. Thus the sequence $r_{k,n}$ has a uniform limit coinciding $\mu$-a.e. with the continuous function $f$ and hence the whole sequence is uniformly converging to $f$, being the two continuous function equal on a carrier of $\mu$ which needs to be dense in $K=\support \mu$.

Now notice, as above, that
\begin{align*}
&\|f-r_{k,n}\|_K\leq \left\|\sum_{j=k+1}^\infty r_{j,n}^\mu-r_{j-1,n}^\mu\right\|_K \leq\sum_{j=k+1}^\infty \|r_{j,n}^\mu-r_{j-1,n}^\mu\|_K\\
\leq& M\mu(K)^{1/2}C(1+\rho) \sum_{j=k+1}^\infty\left(\frac{1+\epsilon}{\rho}\right)^j= M\mu(K)^{1/2}C(1+\epsilon)\frac{1+\rho}{\rho-1}\left(\frac{1+\epsilon}{\rho}\right)^k.
\end{align*}  
Therefore we have
$$\limsup_k \|f-r_{k,n}\|_K^{1/k}\leq \limsup_k \left(\frac{M\mu(K)^{1/2}C(1+\epsilon)(1+\rho)}{\rho-1}\right)^{1/k}\frac{1+\epsilon}{\rho}=\frac{1+\epsilon}{\rho}.$$
The thesis follows letting $\epsilon\to 0^+$ and $\rho\to r^+$.
\qed
\end{proof}

\section*{Acknowledgements}
The author deeply thanks Norman Levenberg (Indiana University, Bloomington \textsc{IN}, USA) who supervised this work for his help. Part of the present research has been carried out during a visit to Universit\'e Aix-Marseille, Marseille FR. The author thanks Frank Wielonsky who made this possible and all people from the Department of Mathematics. The anonymous referees pointed out some important remarks, corrections and suggested Example 1 (d), the authors thanks them for the valuable comments.

\bibliographystyle{abbrv}      
\bibliography{references}

\begin{thebibliography}{10}

\bibitem{AlTa84}
H.~J. Alexander and B.~A. Taylor.
\newblock Comparison of two capacities in $\mathbb{C}^n$.
\newblock {\em Math. Z.}, 186:407--414, 1984.

\bibitem{BeTa82}
E.~Bedord and B.~A. Taylor.
\newblock A new capacity for plurisubharmonic functions.
\newblock {\em Acta Mathematica}, 149(1):1--40, 1982.

\bibitem{BB11b}
R.~Berman and S.~Boucksom.
\newblock Growth of balls of holomorphic sections and energy at equilibrium.
\newblock {\em Invent. Math.}, 181(2):337--394, 2010.

\bibitem{BB11}
R.~Berman, S.~Boucksom, and D.~W. Nymstrom.
\newblock Fekete points and convergence toward equilibrium on complex
  manifolds.
\newblock {\em Acta. Mat.}, 207:1--27, 2011.

\bibitem{Bi99}
P.~Billingsley.
\newblock {\em Convergence of probability measures}.
\newblock John Wiley \& Sons. Wiley Series in probability and statistics, 1999.

\bibitem{Bl09}
H.~P. Blatt.
\newblock The impact of poles on the convergence of best rational approximants.
\newblock {\em Journal of Contemporary Mathematical Analysis}, 44(4):243--251,
  2009.

\bibitem{B97}
T.~Bloom.
\newblock Orthogonal polynomials in $\mathbb{C}^{n}$.
\newblock {\em Indiana University Mathematical Journal}, 46(2):427--451, 1997.

\bibitem{BlLe99}
T.~Bloom and N.~Levenberg.
\newblock Capacity convergence results and applications to a {B}ernstein
  {M}arkov {I}nequality.
\newblock {\em Trans. of AMS}, 351(12):4753--4767, 1999.

\bibitem{BLOLE07}
T.~Bloom and N.~Levenberg.
\newblock Transfinite diameter notions in $\mathbb{C}^n$ and integrals of
  {V}andermonde determinants.
\newblock {\em Ark. Math.}, 48(1):17--40, 2010.

\bibitem{BlLe13}
T.~Bloom and N.~Levenberg.
\newblock Pluripotential energy and large deviation.
\newblock {\em Indiana Univ. Math. J.}, 62(2):523--550, 2013.

\bibitem{BlLeWi13}
T.~Bloom, N.~Levenberg, and F.~Wielonsky.
\newblock Vector energy and large deviations.
\newblock {\em J. Anal. Math}, 125:139--174, 2015.

\bibitem{BlSh07}
T.~Bloom and B.~Shiffman.
\newblock Zeros of random polynomials on $\mathbb{C}^m$.
\newblock {\em Math. Res. Lett.}, 14(3):469--479, 2007.

\bibitem{Go75}
A.~A. Gonchar.
\newblock On a theorem of {S}aff.
\newblock {\em (Russian) Mat. Sb.}, 94(136):152--157, 1975.

\bibitem{BeCa13}
M.~B. Hernándeza and B.~de~la Calle~Ysernb.
\newblock Meromorphic continuation of functions and arbitrary distribution of
  interpolation points.
\newblock {\em Journal of Mathematical Analysis and Applications},
  403(1):107--119, 1971.

\bibitem{Hi}
E.~Hille.
\newblock {\em Analytic function theory. Vol. II}.
\newblock Ginn and Co., Boston, Mass.-New York-Toronto, Ont., 1962.

\bibitem{Kli}
M.~Klimek.
\newblock {\em Pluripotential Theory}.
\newblock Oxford Univ. Press, 1991.

\bibitem{levnotes}
N.~Levenberg.
\newblock Ten lectures on weighted pluripotential theory.
\newblock {\em Dolomites Notes on Approximation}, 5:1--59, 2012.

\bibitem{NqPh12}
D.~N. Quang and H.~P. Hoang.
\newblock Weighted {B}ernstein-{M}arkov property in $\mathbb{C}^n$.
\newblock {\em Ann. Polon. Math.}, 105(2):101--123, 2012.

\bibitem{Rans}
T.~Ransford.
\newblock {\em Potential Theory in the Complex Plane}.
\newblock Cambridge Univ. Press, 1995.

\bibitem{Sa71}
E.~Saff.
\newblock Regions of meromorphy determined by the degree of best rational
  approximation.
\newblock {\em Proc. Amer. Math. Soc.}, (29):30–38, 1971.

\bibitem{Sa10}
E.~B. Saff.
\newblock Logarithmic potential theory with applications to approximation
  theory.
\newblock {\em Surveys in Approximation Theory}, 5:165--200, 2010.

\bibitem{SaTo97}
E.~B. Saff and V.~Totik.
\newblock {\em Logarithmic potentials with external fields}.
\newblock Springer-Verlag Berlin, 1997.

\bibitem{StaTo92}
H.~Stahl and V.~Totik.
\newblock {\em General Orthogonal Polynomials}.
\newblock Cambridge Univ. Press, 1992.

\bibitem{WAL}
J.~L. Walsh.
\newblock {\em Interpolation and approximation by rational function on complex
  domains}.
\newblock AMS, 1929.

\bibitem{Wal35}
J.~L. Walsh.
\newblock {\em Interpolation and approximation by rational functions}.
\newblock colloquium pubblications XX, 1935.

\end{thebibliography}

\end{document}